\documentclass[reqno]{amsart}
\usepackage{amsmath}
\usepackage{amssymb}
\usepackage{cite}
\usepackage{hyperref}
\usepackage{soul}
\hypersetup{ colorlinks = true, urlcolor = blue, linkcolor = blue, citecolor = red }

\usepackage{xcolor}
\usepackage{enumitem}
\usepackage{xparse}
\let\realItem\item % save a copy of the original item
\makeatletter
\NewDocumentCommand\myItem{ o }{%
   \IfNoValueTF{#1}%
      {\realItem}% add an item
      {\realItem[#1]\def\@currentlabel{#1}}% add an item and update label
}
\makeatother

\usepackage{enumitem}    
\setlist[enumerate]{
    before=\let\item\myItem,       % use \myItem in enumerate
    label=\textnormal{(\arabic*)}, % format the label
    widest=(2')                    % set the widest label
}
\usepackage{esint}
\makeatletter
\def\namedlabel#1#2{\begingroup
	#2%
	\def\@currentlabel{#2}%
	\phantomsection\label{#1}\endgroup
}
\usepackage{varwidth}
\usepackage{tasks}
\DeclareMathOperator{\dv}{div}
\DeclareMathOperator{\loc}{loc}

\newcommand{\RR}{\mathbb{R}}
\newcommand{\mA}{\mathcal{A}}
\newcommand{\Om}{\Omega}
\newcommand{\na}{\nabla}
\newcommand{\pa}{\partial}
\newcommand{\La}{\Lambda}
\newcommand{\al}{\alpha}

\newcommand{\ep}{\epsilon}

\newcommand{\la}{\lambda}
\newcommand{\sig}{\sigma}
\newcommand{\cv}{\kappa}
\newcommand{\law}{\lambda_w}
\newcommand{\lav}{\lambda_v}
\newcommand{\data}{\mathit{data}}

\newcommand{\mQ}{U}

\newcommand{\cc}{\textbf{c}_c}
\newcommand{\ccc}{\textbf{c}_{cc}}
\newcommand{\ca}{c_a}

\newcommand{\Caa}{\mathcal{C}^{\alpha,\frac\alpha 2}}

\definecolor{darkgreen}{rgb}{0.00, 0.50, 0.00}

\theoremstyle{plain}
\newtheorem{theorem}{Theorem}[section]
\newtheorem{lemma}[theorem]{Lemma}

\newtheorem{definition}[theorem]{Definition}
\newtheorem{proposition}[theorem]{Proposition}

\newtheorem{remark}[theorem]{Remark}
\def\Xint#1{\mathchoice
	{\XXint\displaystyle\textstyle{#1}}%
	{\XXint\textstyle\scriptstyle{#1}}%
	{\XXint\scriptstyle\scriptscriptstyle{#1}}%
	{\XXint\scriptstyle\scriptscriptstyle{#1}}%
	\!\int}
\def\XXint#1#2#3{{\setbox0=\hbox{$#1{#2#3}{\int}$}
		\vcenter{\hbox{$#2#3$}}\kern-.5\wd0}}

\def\Yint#1{\mathchoice
	{\YYint\displaystyle\textstyle{#1}}%
	{\YYint\textstyle\scriptstyle{#1}}%
	{\YYint\scriptstyle\scriptscriptstyle{#1}}%
	{\YYint\scriptscriptstyle\scriptscriptstyle{#1}}%
	\!\iint}
\def\YYint#1#2#3{{\setbox0=\hbox{$#1{#2#3}{\iint}$}
		\vcenter{\hbox{$#2#3$}}\kern-.51\wd0}}
\def\longdash{{-}\mkern-3.5mu{-}} 
\def\fiint{\Yint\longdash}

\def\Xint#1{\mathchoice
	{\XXint\displaystyle\textstyle{#1}}%
	{\XXint\textstyle\scriptstyle{#1}}%
	{\XXint\scriptstyle\scriptscriptstyle{#1}}%
	{\XXint\scriptscriptstyle\scriptscriptstyle{#1}}%
	\!\int}
\def\XXint#1#2#3{{\setbox0=\hbox{$#1{#2#3}{\int}$ }
		\vcenter{\hbox{$#2#3$ }}\kern-.6\wd0}}

\def\dashint{\Xint-}
\usepackage{nameref}
\makeatletter
\let\orgdescriptionlabel\descriptionlabel
\renewcommand*{\descriptionlabel}[1]{%
	\let\orglabel\label
	\let\label\@gobble
	\phantomsection
	\edef\@currentlabel{#1}%
	\let\label\orglabel
	\orgdescriptionlabel{#1}%
}
\makeatother
\numberwithin{equation}{section}
\author{Iwona Chlebicka}
\address[Iwona Chlebicka]{University of Warsaw, ul. Banacha 2, 02-097 Warsaw, Poland}
\email[Corresponding author]{i.chlebicka@mimuw.edu.pl}

\author{Prashanta Garain}
\address[Prashanta Garain]{Department of Mathematical Sciences,
Indian Institute of Science Education and Research Berhampur,
Berhampur, Permanent Campus, At/Po:-Laudigam,
Dist.-Ganjam, Odisha, India-760003}
\email{pgarain92@gmail.com}

\author{Wontae Kim}
\address[Wontae Kim]{Korea Institute for Advanced Study, 5 Hoegi-ro, Dongdaemun-gu, Seoul 02455, Republic of Korea}
\email{wontae@kias.re.kr}

\def\Xint#1{\mathchoice
    {\XXint\displaystyle\textstyle{#1}}%
    {\XXint\textstyle\scriptstyle{#1}}%
    {\XXint\scriptstyle\scriptscriptstyle{#1}}%
    {\XXint\scriptscriptstyle\scriptscriptstyle{#1}}%
    \!\int}
\def\XXint#1#2#3{\setbox0=\hbox{$#1{#2#3}{\int}$}
    \vcenter{\hbox{$#2#3$}}\kern-0.5\wd0}
\def\fint{\Xint-}
\def\dashint{\Xint{\raise4pt\hbox to7pt{\hrulefill}}}

\def\XXiint#1#2#3{\setbox0=\hbox{$#1{#2#3}{\iint}$}
    \vcenter{\hbox{$#2#3$}}\kern-0.5\wd0}

\begin{document}
	\thanks{I.C. is supported by grant funded by National Science Centre, Narodowe Centrum Nauki,  2019/34/E/ST1/00120. W.K. has been supported by KIAS Individual Grant (HP105501)}
    \thanks{Data sharing not applicable to this article as no datasets were generated or analysed during the current study.}
\title[Gradient higher integrability for double-phase systems]{Gradient higher integrability\\ of bounded solutions\\ to parabolic double-phase systems}

\everymath{\displaystyle}

\makeatletter
\@namedef{subjclassname@2020}{\textup{2020} Mathematics Subject Classification}
\makeatother

\begin{abstract} We prove that bounded solutions  to degenerate parabolic double-phase problem modelled upon
\[u_t-\dv(|\na u|^{p-2}\na u+a(x,t)|\na u|^{q-2}\na u)=-\dv(|F|^{p-2}F+a(x,t)|F|^{q-2}F)\,, \]
where a nonnegative weight $a$ is $\alpha$-H\"older continuous in space and $\tfrac \alpha 2$-H\"older continuous in time, have locally higher integrable gradients for the sharp range of exponents $p<q\le p+\alpha$. 
\end{abstract}

\keywords{Parabolic double-phase systems, parabolic $p$-Laplace systems, gradient estimates}
\subjclass[2020]{35D30, 35K55, 35K65}
\maketitle
%\tableofcontents

\section{Introduction} Parabolic equations with double-phase growth represent a natural parabolic counterpart of models describing materials with heterogeneous hardening, composite media, or diffusion processes with switching intensities. Mathematical description of them have attracted considerable attention over the past decade~\cite{mira,C-b,H-book}. Even for strongly nonlinear problems, one typically expects solutions to exhibit regularity beyond that guaranteed by mere membership in the energy space; see, for instance,~\cite{ElMe,Str,KiLe-Duke}. Recently, there has been remarkable progress in the regularity theory for parabolic double-phase problems, such as \[u_t-\dv(|\na u|^{p-2}\na u+a(x,t)|\na u|^{q-2}\na u)=-\dv(|F|^{p-2}F+a(x,t)|F|^{q-2}F)\,. \] In a short time, it has led to several deep and influential regularity results including~\cite{WMS,WS,sen2025lipschitzregularityparabolicdouble,Sen2025,MR4926910,K-CZ,KKM,KIM2025110738, buskr,ssy}. Some results in a more refined framework are also available, 
 see \cite{oh2025gradient,MR4926910,MR4774133,Sen2025}. Nevertheless, the theory is still far from being fully understood.  
 
 In this work, we focus on gradient higher integrability, which plays a decisive role in the analysis of finer regularity of weak solutions. Solutions to double phase problems are expected be regular provided the closeness condition on the exponents is controlled by the regularity of the weight $a$, which broadens under a priori knowledge about the regularity of $u$. This is observed in the elliptic case~\cite{comi-bdd,bacomi-cv,ok-na}, but its parabolic counterpart remains largely unexplored. We study the regularity to parabolic double phase problems for a priori bounded solutions, where the evolution brings deep complications absent in the elliptic situation. Although related results are available for problems of similar type~\cite{BoDu-nonst,DiBFr,KiLe-Duke,KKM}, the structure of the system considered here prevents an application of the existing techniques. What is more, once we finished the first draft of our manuscript we learned about~\cite{kim2025boundedsolutionsinterpolativegap} where the method cannot embrace the presence of the non-zero right-hand side.   Henceforth, we deliver a substantially new approach.

Let us present in detail our result. We shall consider weak solutions to the parabolic double-phase system
\begin{align}\label{eq}
	u_t-\dv\mA(z,\na u)=-\dv(|F|^{p-2}F+a(z)|F|^{q-2}F)\quad \text{in $\Omega_T=\Omega\times(0,T)$\,,}
\end{align}
 where $\Omega$ is a bounded domain in $\mathbb{R}^n$, $n\geq 2$, and $T>0$. Here, 
we assume $\mA(z,\na u):\Omega_T\times \RR^{Nn}\longrightarrow \RR^{Nn}$ with $N\ge1$ is a Carath\'eodory vector field satisfying the following structure assumptions: there exist constants $0<\nu\le L<\infty$ such that
	\begin{align}\label{str}
	    \mA(z,\xi)\cdot \xi\ge \nu(|\xi|^p+a(z)|\xi|^q)
        \quad\text{and}\quad
			|\mA(z,\xi)|\le L(|\xi|^{p-1}+a(z)|\xi|^{q-1}),
	\end{align}
for almost every $z\in \Omega_T$ and every $\xi\in \RR^{Nn}$. Throughout the rest of the paper, we use the notation for $z\in\Omega_T$ and $s\geq 0$ \begin{equation}
    \label{H-def}
H(z,s)=s^p+a(z) s^q\,.
\end{equation}

We focus of gradient higher integrability of the solutions in the spirit of \cite{KiLe-Duke}. We prove it for a priori bounded $u$, $a \geq 0$ and $a\in \Caa(\Omega_T)$ for some $\al \in (0,1]$ and 
\begin{align} \label{range_pq}
2\le p <\infty, \quad p<q\le p+\alpha\,.    
\end{align}
Here $a\in  \Caa(\Omega_T)$ means that $a\in L^{\infty}(\Omega_T)$ and there exist constants $\ca>0$, such that
\begin{align}
	\label{def_holder}
		|a(x,t)-a(y,s)|\le c_a\left(\max\{ |x-y|^\alpha,|t-s|^\frac{\alpha}{2} \}\right)
\end{align}
for every $(x,y)\in\Omega$ and $(t,s)\in (0,T)$.

The main result of this paper is the following higher integrability estimate for the gradient of a weak solution to \eqref{eq} assuming that the forcing term $H(z,|F|)\in L^\gamma$,
where $\gamma=\tfrac{n+p}{p}$. We denote the constant $c=c(\data)$ if $c$ depends on the following
  \begin{align*}
  \begin{split}
      \mathit{data}=&(n,N,p,q,\nu,L,\ca,\|u\|_{L^\infty(\Om_T)},\|H(z,|F|)\|_{L^\gamma(\Om_T)})\,.
  \end{split}
\end{align*} Now we state our main result.

\begin{theorem}\label{main_theorem}
	Let $u\in C(0,T;L^2(\Om,\RR^N))\cap L^q(0,T;W^{1,q}(\Om,\RR^N))\cap L^\infty(\Omega_T)$ be a weak solution to \eqref{eq} in $\Omega_T$. 
	There exist constants $\ep_0=\ep_0(\mathit{data})$ and $c=c(\mathit{data} , \|a\|_{L^\infty(\Om_T)})$ such that for every $Q_{4r}(z_0)\subset\Om_T$ with $r\in (0,1)$ and $\ep\in(0,\ep_0)$,
	\begin{align*}
    \begin{split}
        &\iint_{Q_{r}(z_0)}H(z,|\na u|)^{1+\epsilon}\,dz\\
	&\le c\left(\left( \frac{\|u\|_{L^\infty(\Om_T)}}{r} \right)^p+\|a\|_{L^\infty(\Om_T)}\left( \frac{\|u\|_{L^\infty(\Om_T)}}{r} \right)^q +1\right)^{1+\frac{q\epsilon}{p}}\\
    &\qquad+ c\left( \fiint_{Q_{4r}(z_0)}H(z,|F|)^{1+\ep} \,dz\right)^{1+\frac{q}{2}}\,.
    \end{split}
\end{align*} 
\end{theorem}
For the rest of the paper, whenever we say that $u$ is a weak solution of \eqref{eq} in $\Omega_T$, we assume that $u\in C(0,T;L^2(\Om,\RR^N))\cap L^q(0,T;W^{1,q}(\Om,\RR^N))\cap L^\infty(\Omega_T)$. See Section~\ref{ssec:sol} for more comments.

\begin{remark}[Sharpness]\rm 
    The question on the optimality of a range of parameters for multiple regularity results, was a topic studied in depth on elliptic problems~\cite{Zhikov86,eslemi}, providing that under a priori boundedness assumption~\eqref{range_pq} is actually sharp, cf.~\cite{comi,comi-bdd,bacomi-cv,fomami,bcdfm}. From this point of view, the current paper is a natural spin-off of~\cite{WMS}, where  gradient higher integrability of the solutions to~\eqref{eq} was proven under no extra a priori regularity of a solution itself.
\end{remark}

{\bf Methods.} The main idea is typical for the double phase problems -- we use the phase separation to regions when the contribution of $a$ is small or dominating. As much as employing the exit time reasoning is expected, a key point is to exploit a delicate relation between the radii of the exit time balls and the level at which we perform the exit time argument, where the stopping time depends on the energy of $\nabla u$ and of $F$. We face several new challenges in the proof of reverse H\"older-type result (in Section~\ref{sec:rev-hold}) comparing to the previous analysis. In fact, the phase separation in  \cite{KKM}, i.e., in the case of non a priori bounded solutions, relied on the  comparison of two values $\la^p$ and $a(z_0)\la^q$ by using just the energy of the gradient of the solutions. In contrast, the phase analysis here is formulated on integrated quantities, allowing us to absorb the $q$-phase into the $p$-phase inside appropriate intrinsic cylinders and vice versa. Then the energy control (including the Caccioppoli type inequality) can be considered as perturbation corresponding estimates known for the $p$-Laplace systems.   We shall point out that the analysis requires delicate interplay between the homogeneity of the scaling factors of cubes and their radii.

A major challenge is obtaining a variant of reverse H\"older inequality within the intrinsic and locally varying scaling, that forms the analytic backbone for the higher integrability proved in Theorem~\ref{main_theorem}. We prove the reversed H\"older inequality  adapted to the intrinsic geometry distinguishing {\it $p$-intrinsic case}, when our problem is a perturbed $p$-Laplace evolution and a perturbation is controlled by  $\|u\|_{L^\infty}$. Namely, for some $K=K(\data)>1$ and relevant cube $Q_\rho$ it holds
\[K\geq (\|u\|_{\infty}/\rho)^{q-p}\sup_{cQ_{\rho}} a\,.\]
On the other hand, when this condition fails, we deal with the {\it $(p,q)$-intrinsic case}, when our problem resembles the $q$-Laplace problem. The reversed H\"older inequality is provided  for $p$-phase in Lemma~\ref{p_reverse_lem} and for $(p,q)$-phase in Lemma~\ref{q_reverse_lem}. 

This separation is effective in the presence of $F$ under the sharp range of phase parameters~\eqref{range_pq}. The proof is possible due to Lemma~\ref{lemma_decay} yielding a key decay property enabling the datum $F$ to be nonzero.

The arguments are closed with use of the Vitali covering theorem, see Lemma~\ref{vitali_lem}, and consequences of the reverse H\"older inequality provided for  $p$-phase in Proposition~\ref{p_est_vitali} and for $(p,q)$-phase in Proposition~\ref{q_est_vitali}. \textcolor{black}{Let us stress that since intrinsic geometries may vary from point to point, the standard Vitali covering lemma cannot be directly applied. The required comparability of the scaling factors is ensured by the Hölder continuity of the coefficient $a$ in Lemma~\ref{la_comp_lem}, which allows us to establish a modified Vitali covering lemma with a scaled covering constant.} The admissible range $q\le p+\alpha$ is precisely what guarantees that the oscillation of $a(\cdot)$ over the cylinders is below the threshold needed for the proof to be valid, see Section~\ref{stopping time argument}.\newline

{\bf Organization. } In Section~\ref{sec:prelim} we present notation and basic definitions. Section~\ref{sec:energy-estimates} provides a priori estimates, while Section~\ref{sec:rev-hold} is devoted to the reverse H\"older estimate. The main proof is concluded in Section~\ref{sec:main-proof}.\newline

\section{Preliminaries}\label{sec:prelim}
We introduce the following notation that will be used throughout this paper.
\subsection{Notation}
We denote a point in $\RR^{n+1}$ as $z=(x,t)$, where $x\in \RR^n$ and $t\in \RR$.
A ball with center $x_0\in\RR^n$ and radius $\rho>0$ is denoted as
\[
    B_\rho(x_0):=\{x\in \RR^n:|x-x_0|<\rho\}\,.
\]
Parabolic cylinders with center $z_0=(x_0,t_0)$ and quadratic scaling in time are denoted as
\[
    Q_\rho(z_0):=B_\rho(x_0)\times I_\rho(t_0)\,,\quad \text{where }\ I_\rho(t_0):=(t_0-\rho^2,t_0+\rho^2)\,.
\]
In this paper, we use two types of intrinsic cylinders. For $\la\geq1$ and $\rho > 0$, a $p$-intrinsic cylinder centered at $z_0=(x_0,t_0)$ is 
\begin{align*}%\label{def_Q_cylinder}
	Q_\rho^\la(z_0):= 
     B_\rho(x_0)\times I^\la_{\rho}(t_0)\,,\quad \text{where }\ I^\la_\rho(t_0) := I_{\la^{\frac{2-p}{2}}\rho}(t_0)\,, 
\end{align*}
and a $(p,q)$-intrinsic cylinders centered at $z_0=(x_0,t_0)$ is
\begin{align*}%\label{def_G_cylinder}
    G_{\rho}^\la(z_0):=
        B_{\rho}(x_0)\times J_{\rho}^\la(t_0)\,, \quad \text{where }\ 
J_\rho^{\la}(t_0):=I_{\la (H(z_0,\la))^{-1}\rho}(t_0)\,
\end{align*}
for $H$ given by~\eqref{H-def}. For $c>0$, we write
\[
    cQ_\rho^\la(z_0)=Q_{c\rho}^\la(z_0)
    \quad\text{and}\quad 
    cG_\rho^\la(z_0)=G_{c\rho}^\la(z_0)\,.
\]
We also consider parabolic cylinders with arbitrary scaling in time and denote
 \[
     Q_{R,\ell}(z_0):=B_R(x_0)\times I_{\ell}(t_0),\quad R,\ell>0\,.
 \]
 
 The $(n+1)$-dimensional Lebesgue measure of a set $E\subset\RR^{n+1}$ is denoted as $|E|$.
For $f\in L^1(\Omega_T,\RR^N)$ and a measurable set $E\subset\Om_T$ with $0<|E|<\infty$, we denote the integral average of $f$ over $E$ as
\[
	(f)_{E}:=\frac{1}{|E|}\iint_{E}f\,dz=\fiint_{E}f\,dz\,.
\]
\subsection{Auxiliary results}
\begin{lemma}\label{tech_lem} 
	Let $0<r<R<\infty$ and $h:[r,R]\longrightarrow\RR$ be a non-negative and bounded function. Suppose there exist $\vartheta\in(0,1)$, $A,B\ge0$ and $\gamma>0$ such that
	\[
		h(r_1)\le \vartheta h(r_2)+\frac{A}{(r_2-r_1)^\gamma}+B
		\quad\text{for all}\quad
		0<r\le r_1<r_2\le R\,.
	\]
	Then there exists a constant $c=c(\vartheta,\gamma)$, such that
	\[
		h(r)\le c\left(\frac{A}{(R-r)^\gamma}+B\right)\,.
	\]
\end{lemma}
We make use of the following Gagliardo--Nirenberg lemma, see \cite[Lemma~8.3]{Giusti}.
\begin{lemma}\label{sobolev_lem}
	Let $B_{\rho}(x_0)\subset\RR^n$, $\sig,s,r\in[1,\infty)$ and $\vartheta\in(0,1)$ such that 
	\[
		-\frac{n}{\sig}\le \vartheta\left(1-\frac{n}{s}\right)-(1-\vartheta)\frac{n}{r}\,.
	\]
	Then for every $v\in W^{1,s}(B_{\rho}(x_0))$ there exists a constant $c=c(n,\sig)$ such that
	\[
		\fint_{B_{\rho}(x_0)}\frac{|v|^\sig}{\rho^\sig}\,dx
		\le c\left(\fint_{B_{\rho}(x_0)}\left(\frac{|v|^s}{\rho^s}+|\na v|^s\right)\,dx\right)^\frac{\vartheta \sig}{s}\left(\fint_{B_{\rho}(x_0)}\frac{|v|^r}{\rho^r}\,dx\right)^\frac{(1-\vartheta)\sig}{r}\,.
	\]
\end{lemma}

\subsection{Solutions}\label{ssec:sol}
The notion of weak solutions to \eqref{eq} is defined as follows.
\begin{definition}%\label{weak_solution}
	A map $u:\Om_T\longrightarrow\RR^N$ satisfying
 \[
			u\in C(0,T;L^2(\Om,\RR^N))\cap L^q(0,T;W^{1,q}(\Om,\RR^N))
	\]
is a weak solution to \eqref{eq}, if 
for every $\varphi\in C_0^\infty(\Omega_T,\RR^N)$ it holds
\begin{align*}
\begin{split}
    &\iint_{\Om_T}(-u\cdot\varphi_t+\mA(z,\na u)\cdot \na \varphi)\,dz\\
    &=\iint_{\Om_T}(|F|^{p-2}F\cdot\na \varphi+a(z)|F|^{p-2}F\cdot \na\varphi)\,dz\,.
\end{split}
\end{align*}
\end{definition}

 We note that to our best knowledge that despite the existence of weak solutions to~\eqref{eq} is expected in the full scope of our current inhnomogeneous in time and space study~\eqref{range_pq}, it is not yet established. We refer to~\cite{cgzg,C-b} for existence to problems with more general growth, but covering probably not sharp regularity with respect to the time variable, and \cite{bgs-discontinuous} for the existence to related problems, that do not fully embrace our case, but remarkably relax demanded time regularity.  At this place we shall stress that in the steady case, due to~\cite{eslemi}, if for $u$ being a weak solution or a local minimizer to related variational functionals, $H(x,\nabla u)\in L^1$, $a\in C^{0,\alpha}$ and $p$ and $q$ satisfy some closeness condition governed by $\alpha$, then $u\in L^q$. Surprisingly, a counterpart of such a result in a parabolic setting is still missing. Nonetheless, we expect it holds true and we continue our work for $L^q(0,T;W^{1,q}(\Om,\RR^N))$-solutions. As for the parabolic approaches under the natural energy regime  $H(z,\nabla u)\in L^1$, let us refer to~\cite{KIM2025110738},

\section{Energy estimates}\label{sec:energy-estimates}
In this section, we provide energy estimates. The first estimate is the Caccioppoli inequality.
\begin{lemma}\label{caccio_lem}
	Let $u$ be a weak solution to \eqref{eq} in $\Omega_T$. Then for every $Q_{R,\ell}(z_0)\subset\Omega_T$, with $R,\ell>0$, and for $r\in [R/2,R)$ and $\tau^2\in[\ell^2/2^2,\ell^2)$, there exists a constant $c=c(n,p,q,\nu,L)$, such that
	\begin{align*}
		\begin{split}
			&\sup_{t\in I_{\tau}(t_0)}\fint_{B_{r}(x_0)}\frac{|u-u_{Q_{r,\tau}(z_0)}|^2}{\tau^2}\,dx+\fiint_{Q_{r,\tau}(z_0)}(|\na u|^p+a(z)|\na u|^q)\,dz\\
			&\le c\fiint_{Q_{R,\ell}(z_0)}\left(\frac{|u-u_{Q_{R,\ell}(z_0)}|^p}{(R-r)^p}+a(z)\frac{|u-u_{Q_{R,\ell}(z_0)}|^q}{(R-r)^q}\right)\,dz\\
			&\qquad+c\fiint_{Q_{R,\ell}(z_0)}\frac{|u-u_{Q_{R,\ell}(z_0)}|^2}{\ell^2-\tau^2}\,dz+c\fiint_{Q_{R,\ell}(z_0)} (|F|^p+a(z)|F|^q)  \,dz\,.
		\end{split}
	\end{align*}
\end{lemma}

\begin{proof}
	Let $\eta\in C_0^\infty(B_{R}(x_0))$ be a cut-off function in the spatial direction satisfying
	\begin{align}\label{sec3:1}
		0\le \eta\le 1\text{ in }B_R(x_0),\quad\eta\equiv1\text{ in } B_{r}(x_0)\quad\text{and}\quad\lVert \na \eta\rVert_{L^\infty(B_R(x_0))}\le \frac{2}{R-r}\,.
	\end{align}
	For $\tau^2\in[\ell^2/2,\ell^2)$, we take sufficiently small $h_0>0$ so that there exists a cut-off function in the time direction $\zeta\in C_0^\infty(I_{\ell-h_0}(t_0))$ such that
	\begin{align}\label{sec3:3}
			0\le \zeta\le 1 \text{ in }I_{l-h_0}(t_0)\,,\ \  \zeta\equiv1\text{ in } I_{\tau}(t_0)\ \text{ and }\ \lVert\pa_t\zeta\rVert_{L^\infty(I_{l-h_0}(t_0))}\le \frac{3}{\ell^2-\tau^2}\,.
	\end{align}
Meanwhile, we take an arbitrary $t_*\in I_\tau(t_0)$ and $\delta\in(0,h_0)$. We define $\zeta_\delta$ as
\begin{align}\label{sec3:4}
	\zeta_{\delta}(t)=
	\begin{cases}
		1\,,&t\in (-\infty,t_*-\delta)\,,\\
		1-\frac{t-t_*+\delta}{\delta}\,,&t\in[t_*-\delta,t_*]\,.\\
		0\,,&t\in(t_*,\infty)\,.
	\end{cases}
\end{align}

For $h\in(0,h_0)$, we take Steklov averages to \eqref{eq} and deduce
\begin{align}\label{sec3:2}
	\pa_t[u-u_{Q_{R,\ell}(z_0)}]_h-\dv [\mA(\cdot,\na u)]_h\
	=-\dv[|F|^{p-2}F+a|F|^{q-2}F]_h
\end{align}
in $B_{R}(x_0)\times I_{\ell-h}(t_0)$.
On the other hand, the function $\varphi=[u-u_{Q_{R,\ell}(z_0)}]_h\eta^q\zeta^2 \zeta_\delta$ belongs to
\[
 W^{1,2}_0(I_{\ell-h}(t_0);L^2(B_{R}(x_0),\RR^N))\cap L^q(I_{\ell-h}(t_0);W_0^{1,q}(B_{R}(x_0),\RR^N))\,.
\]
Taking $\varphi$ as a test function in \eqref{sec3:2}, we have
\begin{align}\label{sec3:6}
	\begin{split}
		\mathrm{I}+\mathrm{II}&=
		\fiint_{Q_{R,\ell}(z_0)}\pa_t[u-u_{Q_{R,\ell}}(z_0)]_h\cdot\varphi\,dz
		+\fiint_{Q_{R,\ell}(z_0)}[\mA(\cdot,\na u)]_h\cdot\na\varphi \,dz\\
		&=\fiint_{Q_{R,\ell}(z_0)}[|F|^{p-2}F+a|F|^{q-2}F]_h\cdot \na\varphi\,dz=\mathrm{III}\,.
	\end{split}
\end{align}
The integrals $\mathrm{II}$ and $\mathrm{III}$ are finite. 
Indeed, it follows from \eqref{str} and the properties of the Steklov average that
	\begin{align*}
			\begin{split}
			    \mathrm{II}
			&\le L\fiint_{Q_{R,\ell}(z_0)}\left|[|\nabla u|^{p-1}]_h(x,t)\right|\left|\nabla\varphi(x,t)\right|\, dx\ dt\\
            &\qquad+L\fiint_{Q_{R,\ell}(z_0)}|[a|\nabla u|^{q-1}]_h(x,t)| |\nabla\varphi(x,t)|\, dx\ dt\,.
			\end{split}
	\end{align*}
     The first term on the right-hand side is finite as in parabolic $p$-Laplace systems. Meanwhile, the second term on the right-hand side can be written as
	\begin{align*}
     \begin{split}
     &\fiint_{Q_{R,\ell}(z_0)}|[a|\nabla u|^{q-1}]_h(x,t)| |\nabla\varphi(x,t)|\, dx\ dt\\
    & =\fiint_{Q_{R,\ell}(z_0)}\fint_t^{t+h}[a(x,s)]^\frac{q-1}{q}|\nabla u(x,s)|^{q-1}\  [a(x,s)]^\frac{1}{q}|\nabla\varphi(x,t)| \, ds \, dx\ dt\,.
      \end{split}
	\end{align*} 
	Employing H\"older's inequality and the properties of the Steklov average, there exists a constant $c=c(n)$ such that
	\begin{align*}
		\begin{split}
			&\fiint_{Q_{R,\ell}(z_0)}|[a|\nabla u|^{q-1}]_h(x,t)| |\nabla\varphi(x,t)|\, dx\ dt\\
			&\le c\left(\fiint_{Q_{R,\ell}(z_0)}a|\nabla u|^{q}  \ dx\ dt\right)^\frac{q-1}{q}\left(\fiint_{Q_{R,\ell}(z_0)}\fint_t^{t+h}a(x,s)|\nabla \varphi(x,t)|^{q}\,ds\,dx\ dt\right)^\frac{1}{q}\\
            &= c\left(\fiint_{Q_{R,\ell}(z_0)}a|\nabla u|^{q}\, dx\ dt\right)^\frac{q-1}{q}\left(\fiint_{Q_{R,\ell}(z_0)}a_h(x,t)|\nabla \varphi(x,t)|^{q}\, dx\ dt\right)^\frac{1}{q}\,,
		\end{split}
	\end{align*} 
which shows that $\mathrm{II}$ is finite provided $|\na u|\in L^q(\Om_T)$. The same argument applies for the finiteness of $\mathrm{III}$. From below, we estimate each term.

\textbf{Estimate of $\mathrm{I}$:} Applying integration by parts, there holds
\begin{align}\label{sec3:7}
	\begin{split}
		\mathrm{I}
		&=\fiint_{Q_{R,\ell}(z_0)}\frac{1}{2}(\pa_t |[u-u_{Q_{R,\ell}(z_0)}]_h|^2)\eta^q\zeta^2\zeta_{\delta}\,dz\\
		&=-\fiint_{Q_{R,\ell}(z_0)}|[u-u_{Q_{R,\ell}(z_0)}]_h|^2\eta^q\zeta\zeta_{\delta}\pa_t\zeta\,dz\\
		&\qquad-\fiint_{Q_{R,\ell}(z_0)}\frac{1}{2}|[u-u_{Q_{R,\ell}(z_0)}]_h|^2\eta^q\zeta^2\pa_t\zeta_{\delta}\,dz\,.
	\end{split}
\end{align}
For the first term on the right-hand side of \eqref{sec3:7}, we estimate using \eqref{sec3:3}, Then,
\[
		-\fiint_{Q_{R,\ell}(z_0)} | [u-u_{Q_{R,\ell}(z_0)}]_h |^2\eta^q\zeta\zeta_{\delta}\pa_t\zeta\,dz
		\ge-c\fiint_{Q_{R,\ell}(z_0)}\frac{ |[u-u_{Q_{R,\ell}(z_0)}]_h |^2}{\ell^2-\tau^2}\,dz\,.
\]
Regarding the second term  on the right-hand side of \eqref{sec3:7}, we apply \eqref{sec3:4} to have
\begin{align*}
	\begin{split}
		&-\fiint_{Q_{R,\ell}(z_0)}\frac{1}{2} | [u-u_{Q_{R,\ell}(z_0)}]_h | ^2\eta^q\zeta^2\pa_t\zeta_{\delta}\,dz\\
		&=\frac{1}{|Q_{R,\ell}|}\fint_{t_*-\delta}^{t_*}\int_{B_{R}(x_0)}\frac{1}{2} | [u-u_{Q_{R,\ell}(z_0)}]_h |^2\eta^q\zeta^2\,dx\,dt\\
		&\ge \frac{1}{|Q_{R,\ell}|}\fint_{t_*-\delta}^{t_*}\int_{B_{r}(x_0)}\frac{1}{2} | [u-u_{Q_{R,\ell}(z_0)}]_h | ^2\,dx\,dt\,.
	\end{split}
\end{align*}
Since both $u-u_{Q_{R,\ell}(z_0)}$ and $[u-u_{Q_{R,\ell}(z_0)}]_h$ lie in $C(I_{\ell-h}(t_0);L^2(B_R(x_0),\mathbb{R}^N ))$, the above integral over ball and time interval converges to the integral over the ball at the time $t=t_*$ by the Lebesgue point theorem as $\delta$ goes to $0^+$, and moreover $[u-u_{Q_{R,\ell}(z_0)}]_h$ converges to $u-u_{Q_{R,\ell}(z_0)}$ in the norm of $C(I_{\ell}(t_0);L^2(B_R(x_0),\mathbb{R}^N ))$ as $h$ goes to $0^+$. Therefore, we obtain
\begin{align*}
	\begin{split}
		\lim_{h\to0^+}\lim_{\delta\to0^+}\mathrm{I}&\ge -c\fiint_{Q_{R,\ell}(z_0)}\frac{|u-u_{Q_{R,\ell}(z_0)}|^2}{\ell^2-\tau^2}\,dz\\
		&\qquad+\frac{1}{2|Q_{R,\ell}|}\int_{B_r(x_0)}|u(x,t_*)-u_{Q_{R,\ell}(z_0)}|^2\,dx\,.
	\end{split}
\end{align*}

\textbf{Estimate of $\mathrm{II}$:} Since Steklov averages are involved in the time direction, we have
\begin{align}\label{sec3:11}
	\begin{split}
		\mathrm{II}
		&=\fiint_{Q_{R,\ell}(z_0)}[\mA(\cdot,\na u)]_h\cdot[ \na u]_h\eta^q\zeta^2\zeta_\delta\,dz\\
		&\qquad+q\fiint_{Q_{R,\ell}(z_0)}[\mA(\cdot,\na u)]_h\cdot [u-u_{Q_{R,\ell}(z_0)}]_h\na\eta\eta^{q-1}\zeta^2\zeta_\delta\,dz\,.
	\end{split}
\end{align}
To estimate the first term in \eqref{sec3:11}, we apply properties of Steklov averages and \eqref{str}, Then we get
\begin{align*}
	\begin{split}
			&\lim_{h\to0^+}\lim_{\delta\to0^+}\fiint_{Q_{R,\ell}(z_0)}[\mA(\cdot,\na u)]_h\cdot [\na u]_h\eta^q\zeta^2\zeta_\delta\,dz	\\
			&\ge \frac{\nu}{|Q_{R,\ell}|}\int_{I_{\ell}(t_0)\cap (-\infty,t_*)}\int_{B_{R}(x_0)}(|\na u|^p+a(z)|\na u|^q)\eta^q\zeta^2\,dx\,dt\,.
	\end{split}
\end{align*}
To estimate the second term in \eqref{sec3:11}, we use \eqref{str} and \eqref{sec3:1}, Then
\begin{align*}
	\begin{split}
		&\lim_{h\to0^+}\lim_{\delta\to0^+}q\fiint_{Q_{R,\ell}(z_0)}[\mA(\cdot,\na u)]_h\cdot [u-u_{Q_{R,\ell}(z_0)}]_h\na\eta\eta^{q-1}\zeta^2\zeta_\delta\,dz\\
		&\ge - \frac{2Lq}{|Q_{R,\ell}|}\int_{I_{\ell}(t_0)\cap (-\infty,t_*)}\int_{B_{R}(x_0)}|\na u|^{p-1}\eta^{q-1}\zeta^2\frac{|u-u_{Q_{R,\ell}(z_0)}|}{R-r}\,dx\,dt\\
		&\qquad-\frac{2Lq}{|Q_{R,\ell}|}\int_{I_{\ell}(t_0)\cap (-\infty,t_*)}\int_{B_{R}(x_0)}a(z)|\na u|^{q-1}\eta^{q-1}\zeta^2\frac{|u-u_{Q_{R,\ell}(z_0)}|}{R-r}\,dx\,dt\,.
	\end{split}
\end{align*}
By applying Young's inequality, we have
\begin{align*}
	\begin{split}
		& \lim_{h\to0^+}\lim_{\delta\to0^+}q\fiint_{Q_{R,\ell}(z_0)}[\mA(z,\na u)]_h\cdot [u-u_{Q_{R,\ell}(z_0)}]_h\na\eta\eta^{q-1}\zeta^2\zeta_\delta\,dz\\
		&\ge-\frac{\nu}{4|Q_{R,\ell}|}\int_{I_{\ell}(t_0)\cap (-\infty,t_*)}\int_{B_{R}(x_0)}(|\na u|^p+a(z)|\na u|^q)\eta^q\zeta^2\,dx\,dt\\
		&\qquad-c\fiint_{Q_{R,\ell}(z_0)}\left(\frac{|u-u_{Q_{R,\ell}(z_0)}|^p}{(R-r)^p}+a(z)\frac{|u-u_{Q_{R,\ell}(z_0)}|^q}{(R-r)^q}\right)\,dz
	\end{split}
\end{align*}
for some $c=c(p,q,\nu,L)$.
It follows that
\begin{align*}
	\begin{split}
		\lim_{h\to0^+}\lim_{\delta\to0^+}\mathrm{II}&\ge \frac{3\nu}{4|Q_{R,\ell}|}\int_{I_{\ell}(t_0)\cap (-\infty,t_*)}\int_{B_{R}(x_0)}(|\na u|^p+a(z)|\na u|^q)\eta^q\zeta^2\,dx\,dt\\
		&\qquad-c\fiint_{Q_{R,\ell}(z_0)}\left(\frac{|u-u_{Q_{R,\ell}(z_0)}|^p}{(R-r)^p}+a(z)\frac{|u-u_{Q_{R,\ell}(z_0)}|^q}{(R-r)^q}\right)\,dz\,.
	\end{split}
\end{align*}

\textbf{Estimate of $\mathrm{III}$:} We apply properties of Steklov averages and Young's inequality as above. Then
\begin{align*}
	\begin{split}
		\lim_{h\to0^+}\lim_{\delta\to0^+}\mathrm{III}
		&\le c\fiint_{Q_{R,\ell}(z_0)}(|F|^p+a(z)|F|^q)\,dz\\
		&\qquad+\frac{\nu}{2|Q_{R,\ell}|}\int_{I_{\ell}(t_0)\cap (-\infty,t_*)}\int_{B_{R}(x_0)}(|\na u|^p+a(z)|\na u|^q)\eta^q\zeta^2\,dx\,dt\\
		&\qquad+c\fiint_{Q_{R,\ell}(z_0)}\left(\frac{|u-u_{Q_{R,\ell}(z_0)}|^p}{(R-r)^p}+a(z)\frac{|u-u_{Q_{R,\ell}(z_0)}|^q}{(R-r)^q}\right)\,dz\,.
	\end{split}
\end{align*}

Therefore combining all these estimates, \eqref{sec3:6} becomes
\begin{align*}
	\begin{split}
		&\frac{1}{|Q_{R,\ell}|}\int_{B_r(x_0)}|u(x,t_*)-u_{Q_{R,\ell}(z_0)}|^2\,dx\\
		&\qquad+\frac{1}{|Q_{R,\ell}|}\int_{I_{\ell}(t_0)\cap (-\infty,t_*)}\int_{B_{R}(x_0)}(|\na u|^p+a(z)|\na u|^q)\eta^q\zeta^2\,dx\,dt\\
		&\le c\fiint_{Q_{R,\ell}(z_0)}\left(\frac{|u-u_{Q_{R,\ell}(z_0)}|^p}{(R-r)^p}+a(z)\frac{|u-u_{Q_{R,\ell}(z_0)}|^q}{(R-r)^q}\right)\,dz\\		
		&\qquad+c\fiint_{Q_{R,\ell}(z_0)}\frac{|u-u_{Q_{R,\ell}(z_0)}|^2}{\ell^2-\tau^2}\,dz+c\fiint_{Q_{R,\ell}(z_0)}(|F|^p+a(z)|F|^q)\,dz\,.
	\end{split}
\end{align*}
Since $t_*\in I_{\tau}(t_0)$ is arbitrary, $|B_{R}|\approx c(n)|B_{r}|$ and $|I_{\ell}|\approx |I_{\tau}|$, we get
\begin{align*}
	\begin{split}
		&\sup_{t\in I_\tau(t_0)}\fint_{B_{r}(x_0)}\frac{|u(x,t)-u_{Q_{R,\ell}(z_0)}|^2}{\tau^2}\,dx+\fiint_{Q_{r,\tau}(z_0)}(|\na u|^p+a(z)|\na u|^q)\,dz\\
		&\le c\fiint_{Q_{R,\ell}(z_0)}\left(\frac{|u-u_{Q_{R,\ell}(z_0)}|^p}{(R-r)^p}+a(z)\frac{|u-u_{Q_{R,\ell}(z_0)}|^q}{(R-r)^q}\right)\,dz\\
		&\qquad+c\fiint_{Q_{R,\ell}(z_0)}\frac{|u-u_{Q_{R,\ell}(z_0)}|^2}{\ell^2-\tau^2}\,dz+c\fiint_{Q_{R,\ell}(z_0)}(|F|^p+a(z)|F|^q)\,dz\,.
	\end{split}
\end{align*}
Finally, since the triangle inequality implies 
\begin{align*}
    \begin{split}
        &\sup_{t\in I_\tau(t_0)}\fint_{B_{r}(x_0)}|u(x,t)-u_{Q_{r,\tau}(z_0)}|^2\,dx\\
        &=\sup_{t\in I_\tau(t_0)}\fint_{B_{r}(x_0)}|u(x,t)-u_{Q_{R,\ell}(z_0)} +u_{Q_{R,\ell}(z_0)}-u_{Q_{r,\tau}(z_0)}|^2\,dx\\
        &\le 2\sup_{t\in I_\tau(t_0)}\fint_{B_{r}(x_0)}|u(x,t)-u_{Q_{R,\ell}(z_0)}|^2\,dx+  2|u_{Q_{r,\tau}(z_0)}-u_{Q_{R,\ell}(z_0)}|^2\\
        &\le 4\sup_{t\in I_\tau(t_0)}\fint_{B_{r}(x_0)}|u(x,t)-u_{Q_{R,\ell}(z_0)}|^2\,dx\,,
    \end{split}
\end{align*}
the proof is completed by substituting this inequality to the left hand side of the previous inequality.
\end{proof}

The second lemma is a gluing result.
\begin{lemma}\label{sec3:lem:2}
	Let $u$ be a weak solution to \eqref{eq} in $\Omega_T$ and $\eta\in C_0^\infty(B_{R}(x_0))$ be a nonnegative function such that  some constant $c>0$ it holds
	\begin{align}\label{sec3:34}
		 \fint_{B_{R}(x_0)}\eta\,dx=1\quad\text{and}\quad \lVert  \eta\rVert_{L^\infty(B_R(x_0))}+R\lVert\na \eta\rVert_{L^\infty(B_R(x_0))}<c\,.
	\end{align} 
	Then for $Q_{R,\ell}(z_0)\subset\Omega_T$ with $R,\ell>0$, there exists a constant $c=c(L)>0$ such that
	\begin{align*}
		\begin{split}
			&\sup_{t_1,t_2\in I_{\ell}(t_0)}|(u\eta)_{B_R(x_0)}(t_2)-(u\eta)_{B_R(x_0)}(t_1)|\\
			&\le c\frac{\ell^2}{R}\fiint_{Q_{R,\ell}(z_0)}(|\na u|^{p-1}+a(z)|\na u|^{q-1})\,dz\\
			&\qquad+c\frac{\ell^2}{R}\fiint_{Q_{R,\ell}(z_0)}(|F|^{p-1}+a(z)|F|^{q-1})\,dz\,.
		\end{split}
	\end{align*}
\end{lemma}

\begin{proof}
	Take arbitrary $t_1,t_2\in I_{\ell}(t_0)$ with $t_1<t_2$. For $\delta\in(0,1)$ small enough, let $\zeta_\delta\in W_0^{1,\infty}(I_{\ell}(t_0))$ be defined as
	\begin{align*}
		\zeta_\delta(t)=
		\begin{cases}
			\qquad 0\,,&t_0-\ell^2\le t\le t_1-\delta\,,\\
			\frac{t-t_1+\delta}{\delta}\,,&t_1-\delta<t<t_1be\,,\\
			\qquad1\,,&t_1\le t\le t_2\,,\\
			\frac{t_2+\delta-t}{\delta}\,,&t_2<t<t_2+\delta\,,\\
			\qquad0\,,&t+\delta\le t\le t_0+\ell^2\,.
		\end{cases}
	\end{align*}
	Taking $\pm\eta\zeta_\delta\in W_0^{1,\infty}(Q_{R,\ell}(z_0))$ as a test function in \eqref{eq}, integration by parts gives
	\begin{align*}
		\begin{split}
			&\mp\fint_{t_1-\delta}^{t_1}\fint_{B_{R}(x_0)} u\eta\,dx\,dt\pm\fint_{t_2+\delta}^{t_2}\fint_{B_{R}(x_0)} u\eta\,dx\,dt\\
			&\le L\int_{t_1-\delta}^{t_2+\delta}\fint_{B_{R}(x_0)} (|\na u|^{p-1}+a(z)|\na u|^{q-1})|\na \eta||\zeta_\delta|\,dx\,dt\\
			&\qquad+\int_{t_1-\delta}^{t_2+\delta}\fint_{B_{R}(x_0)}(|F|^{p-1}+a(z)|F|^{q-1})|\na \eta||\zeta_\delta|\,dz\,.
		\end{split}
	\end{align*}
	Since $u\in C(I_\ell(t_0),L^2(B_{R}(x_0)))$ holds, letting $\delta\to0^+$ with Lebesgue point theorem and using the third condition in \eqref{sec3:34}, we obtain
	\begin{align*}
		\begin{split}
			&|(u\eta)_{B_R(x_0)}(t_1)-(u\eta)_{B_R(x_0)}(t_2)|\\
			&\le c\frac{\ell^2}{R}\fiint_{Q_{R,\ell}(z_0)}(|\na u|^{p-1}+a(z)|\na u|^{q-1})\,dz\\
			&\qquad+c\frac{\ell^2}{R}\fiint_{Q_{R,\ell}(z_0)}(|F|^{p-1}+a(z)|F|^{q-1})\,dz\,.
		\end{split}
	\end{align*}
	This completes the proof.
\end{proof}

We will use the following version of the parabolic Poincar\'e type inequality.
\begin{lemma}\label{poincare_lem}
	Let $u$ be a weak solution to \eqref{eq} in $\Omega_T$. Then for every $Q_{R,\ell}(z_0)\subset\Omega_T$ with $R,\ell>0$, $m\in(1,q]$ and $\theta\in(1/m,1]$, there exists a constant $c=c(n,N,m,L)>0$, such that
	\begin{align*}
    		\begin{split}
			&\fiint_{Q_{R,\ell}(z_0)}\frac{|u-u_{Q_{R,\ell}(z_0)}|^{\theta m}}{R^{\theta m}}\,dz\\
            &\le c\fiint_{Q_{R,\ell}(z_0)}|\na u|^{\theta m}\,dz\\
			&\qquad+c\left(\frac{\ell^2}{R^2}\fiint_{Q_{R,\ell}(z_0)}(|\na u|^{p-1}+a(z)|\na u|^{q-1}+|F|^{p-1}+a(z)|F|^{q-1})
   \,dz\right)^{\theta m}\,.
		\end{split}
	\end{align*}
\end{lemma}

\begin{proof}
	The triangle inequality gives
	\begin{align*}
		\begin{split}
			&\fiint_{Q_{R,\ell}(z_0)}\frac{|u-(u)_{Q_{R,\ell}(z_0)}|^{\theta m}}{R^{\theta m}}\,dz\\
			&\le c\fiint_{Q_{R,\ell}(z_0)}\frac{|u(x,t)-(u)_{B_R(x_0)}(t)|^{\theta m}}{R^{\theta m}}\,dz\\
			&\qquad+c\fint_{I_{\ell}(t_0)}\frac{|(u)_{Q_{R,\ell}(z_0)}-(u)_{B_R(x_0)}(t)|^{\theta m}}{R^{\theta m}}\,dt\,,
		\end{split}
	\end{align*}
	where $c=c(m)$. By applying the Poincar\'e inequality in the spatial direction, we have
	\begin{align*}
		\begin{split}
        &\fiint_{Q_{R,\ell}(z_0)}\frac{|u-(u)_{Q_{R,\ell}(z_0)}|^{\theta m}}{R^{\theta m}}\,dz\\
			&\le c\fiint_{Q_{R,\ell}(z_0)}|\na u|^{\theta m}\,dz%\\			&\qquad
            +c\fint_{I_{\ell}(t_0)}\frac{|(u)_{Q_{R,\ell}(z_0)}-(u)_{B_R(x_0)}(t)|^{\theta m}}{R^{\theta m}}\,dt\,,
		\end{split}
	\end{align*}
where $c=c(n,N,m)$. It remains to estimate the last term. For this, we observe
	\begin{align*}
			\begin{split}
			&\fint_{I_{\ell}(t_0)}|(u)_{Q_{R,\ell}(z_0)}-(u)_{B_R(x_0)}(t)|^{\theta m}\,dt\\
			&\le \fint_{I_{\ell}(t_0)}\fint_{I_{\ell}(t_0)}|(u)_{B_R(x_0)}(\tau)-(u)_{B_R(x_0)}(t)|^{\theta m}\,dt \,d\tau.
			\end{split}
	\end{align*}
	Taking $\eta\in C_0^\infty(B_R(x_0))$ satisfying \eqref{sec3:34}, we get for $c=c(m)$
	\begin{align*}
		\begin{split}
			&\fint_{I_{\ell}(t_0)}\fint_{I_{\ell}(t_0)}|(u)_{B_R(x_0)}(\tau)-(u)_{B_R(x_0)}(t)|^{\theta m}\,dt \,d\tau \\
			&\le c\fint_{I_{\ell}(t_0)}|(u\eta)_{B_R(x_0)}(t)-(u)_{B_R(x_0)}(t)|^{\theta m}\,dt\\
			&\qquad+c\sup_{t,\tau\in I_\ell(t_0)}|(u\eta)_{B_R(x_0)}(t)-(u\eta)_{B_R(x_0)}(\tau)|^{\theta m}\,.
		\end{split}
	\end{align*}
	The last term is estimated by Lemma~\ref{sec3:lem:2}. In the remaining of the proof, we estimate the first term on the right hand side.
	By \eqref{sec3:34} we have
	\begin{align*}
		\begin{split}
			&\fint_{I_{\ell}(t_0)}|(u\eta)_{B_R(x_0)}(t)-(u)_{B_R(x_0)}(t)|^{\theta m}\,dt\\
			&=\fint_{I_{\ell}(t_0)}\left|\fint_{B_{R}(x_0)} (u(x,t)-(u)_{B_R(x_0)}(t))\eta(x)\,dx\right|^{\theta m}\,dt\\
			&\le c \fint_{I_{\ell}(t_0)}\left(\fint_{B_{R}(x_0)} |u(x,t)-(u)_{B_R(x_0)}(t)|\,dx\right)^{\theta m}\,dt\,.
		\end{split}
	\end{align*}
	Therefore, using the Poincar\'e inequality in the spatial direction and H\"older's inequality, we obtain
	\begin{align*}
		\fint_{I_{\ell}(t_0)}|(u\eta)_{B_R(x_0)}(t)-(u)_{B_R(x_0)}(t)|^{\theta m}\,dt\le cR^{\theta m}\fiint_{Q_{R,\ell
        }(z_0)}|\na u|^{\theta m}\,dz\,.
	\end{align*}
	The proof is completed.
\end{proof}

\section{Reverse Hölder inequality}\label{sec:rev-hold}
Let $u$ be a weak solution to \eqref{eq} in $\Omega_T$. In this section, we provide a reverse Hölder inequality for $\na u$. Throughout this section, consider 
\begin{align}\label{upper_level_set}
\Psi(\varkappa)=\{ z\in \Omega_T: H(z,|\na u(z)|)>\varkappa\}\,,
\end{align}
for some 
$\varkappa>1+\|a\|_{L^\infty(\Omega_T)}\geq H(z,1)$ for all $z\in\Om_T$.
Note that $H(z,s)$ is strictly increasing and continuous with respect to nonnegative $s$ variable with
\[
	\lim_{s\to0^+}H(z,s)=0
	\quad\text{and}\quad
	\lim_{s\to\infty}H(z,s)=\infty\,.
\] 
Therefore, by the intermediate value theorem for continuous functions, we find a unique $s=s(z)>1$ such that
\[
	\varkappa=H(z,s)=s^p+a(z)s^q\,.
\]
We let
\begin{align}\label{K_and_kappa}
    K=2+10\ca\|u\|_{L^\infty(\Om_T)}^{q-p}\quad\text{and}\quad\cv = 10(1+\cc^{q}\|u\|_{L^\infty(\Om_T)}^{q-p}+\ca +10\cc\ca )\,,
\end{align}
where $\cc=\cc(n,p,q,\nu,L,\|u\|_{L^\infty(\Om_T)},\|H(z,|F|)\|_{L^\gamma(\Om_T)})$ with $\gamma=\tfrac{n+p}{p}$ is a constant playing an important role in Lemma~\ref{lemma_decay}.

Suppose $z_0\in \Psi(\La)$ with $\La>1+\|a\|_{L^\infty(\Om_T)}$ and there exists $\rho\in(0,1)$ such that $Q_{2\kappa\rho}^\la(z_0)\subset \Om_T$ where $\La=H(z_0,\la)$ for $\la=\la(z_0)$. The reasoning is performed separately for the $p$-intrinsic cylinders and the $( p, q)$-intrinsic cylinders. More precisely, we distinguish
\begin{enumerate}
    \item[{\it (p-1)}]\label{p1} $p$-intrinsic case
    \[
    K\ge\left( \frac{\|u\|_{L^\infty(\Om_T)}}{\rho}  \right)^{q-p}\sup_{z\in Q_{4\rho}^\la(z_0)}a(z)\, ;
    \]
\end{enumerate}
where we set 
\begin{enumerate}
    \item[{\it (p-2)}]the stopping time argument as\label{p2}
    \begin{enumerate}
        \item[{\it (p-2-i)}]\label{p3} $\fiint_{Q_{\rho}^\la(z_0)}
        (
        H(z,|\na u|)
        +H(z,|F|) )
        \,dz= \la^p\,,$
        \item[{\it (p-2-ii)}]\label{p4} $\fiint_{Q_{s}^\la(z_0)}
        (
        H(z,|\na u|)
        +H(z,|F|) )
        \,dz< \la^p$ for every $s\in(\rho,2\cv\rho]\,$. 
    \end{enumerate}
\end{enumerate}
On the other hand, the alternative is the $(p,q)$-intrinsic case, where apart from  the condition complementary to~\ref{p1}, i.e.,  \[
    K\le\left( \frac{\|u\|_{L^\infty(\Om_T)}}{\rho}  \right)^{q-p}\sup_{z\in Q_{4\rho}^\la(z_0)}a(z)\, ;
    \] we assume
\begin{enumerate} 
    \item[{\it (p,q-1)}]\label{q1} the comparability of $a(\cdot)$
    \[
    \frac{a(z_0)}{2\ca}< a(z)< 2\ca  a(z_0) \quad\text{for every} \quad z\in G_{4\rho}^\la(z_0)\,.
    \]
    \item[{\it (p,q-2)}]\label{q2} Stopping time argument for a $(p,q)$-intrinsic cylinder reads
    \begin{enumerate}
        \item[{\it (p,q-2-i)}]\label{q4} $\fiint_{G_{\rho}^\la(z_0)}
        (
        H(z,|\na u|)
        +H(z,|F|) )
        \,dz=\La\,,$
        \item[{\it (p,q-2-ii)}]\label{q5} $\fiint_{G_{s}^\la(z_0)}
        (
        H(z,|\na u|)
        +H(z,|F|) )
        \,dz< \La$ for every $s\in(\rho,2\cv\rho]\,$.
    \end{enumerate}
\end{enumerate}
The proof that upon the range \eqref{range_pq} this set of conditions describes all scenarios is presented in Section~\ref{stopping time argument}.

\subsection{The \texorpdfstring{$p$}{p}-intrinsic case}
In this subsection, we prove the reverse Hölder inequality provided in Lemma~\ref{p_reverse_lem} and its consequence ready to apply in the main proof, namely Proposition~\ref{p_est_vitali}.
We begin with the following decay estimate that take into account the contribution of the nonzero forcing term.
\begin{lemma}\label{lemma_decay}
   Suppose \ref{p1} and \ref{p2}. Then there exists a constant $\cc=\cc(\data)\ge K^\frac{1}{p}$
   such that
   \[
   \rho\le \cc \la^{-1}\,.
   \]
\end{lemma}
\begin{proof}
    We apply Lemma~\ref{caccio_lem} with $r=\rho$, $R=2\rho$, $\tau^2=\la^{2-p}\rho^2$ and $\ell^2=2\la^{2-p}\rho^2$. Then along with \ref{p3}, we have
    	\begin{align*}
		\begin{split}
			\la^p&=\fiint_{Q_{\rho}^\la(z_0)}H(z,|\na u|)\,dz+\fiint_{Q_{\rho}^\la(z_0)}H(z,|F|)\,dz\\
			&\le c\fiint_{Q_{2\rho}^\la(z_0)}\left(\frac{|u-u_{Q_{2\rho}^\la(z_0)}|^p}{\rho^p}+a(z)\frac{|u-u_{Q_{2\rho}^\la(z_0)}|^q}{\rho^q}\right)\,dz\\
			&\qquad+c\la^{p-2}\fiint_{Q_{2\rho}^\la(z_0)}\frac{|u-u_{Q_{2\rho}^\la(z_0)}|^2}{\rho^2}\,dz+c\fiint_{Q_{2\rho}^\la(z_0)}H(z,|F|)\,dz\,,
		\end{split}
	\end{align*}
    where $c=c(n,p,q,\nu,L)$. For the first term on the right hand side, we employ \ref{p1} to obtain
    \begin{align*}
        \begin{split}
            &\fiint_{Q_{2\rho}^\la(z_0)}\left(\frac{|u-u_{Q_{2\rho}^\la(z_0)}|^p}{\rho^p}+a(z)\frac{|u-u_{Q_{2\rho}^\la(z_0)}|^q}{\rho^q}\right)\,dz\\
            &\le c\left( \frac{\|u\|_{L^\infty(\Om_T)}^p}{\rho^p}  +\sup_{z\in Q_{2\rho}^\la(z_0)}a(z)\frac{\|u\|_{L^\infty(\Om_T)}^q}{\rho^q}   \right)\\
            &\le c(1+K)\frac{\|u\|_{L^\infty(\Om_T)}^p}{\rho^p}\,.
        \end{split}
    \end{align*}
    To estimate the second term on the right hand side, we apply Young's inequality and use the fact that $p\ge 2$ to get
    \begin{align*}
        \begin{split}
            c\la^{p-2}\fiint_{Q_{2\rho}^\la(z_0)}\frac{|u-u_{Q_{2\rho}^\la(z_0)}|^2}{\rho^2}\,dz
            &\le c\la^{p-2}\frac{\|u\|^2_{L^\infty(\Om_T)}}{\rho^2} \\
            &\le \frac{1}{4}\la^p+ c\frac{\|u\|^p_{L^\infty(\Om_T)}}{\rho^p}\,.
        \end{split}
    \end{align*}
    For the last term, we apply H\"older's inequality and Young's inequality to have
    \begin{align*}
        \begin{split}
            c\fiint_{Q_{2\rho}^\la(z_0)}H(z,|F|)\,dz
            &\le c\left( \fiint_{Q_{2\rho}^\la(z_0)}H(z,|F|)^\gamma\,dz\right)^\frac{1}{\gamma}\\
            &\le c\la^\frac{p-2}{\gamma}\left(\frac{1}{\rho^{n+2}} 
\iint_{\Om_T}H(z,|F|)^\gamma\,dz  \right)^\frac{1}{\gamma}\\
&\le \frac{1}{4}\lambda^p+ c\left(\frac{1}{\rho^{n+2}} 
\iint_{\Om_T}H(z,|F|)^\gamma\,dz  \right)^\frac{p}{n+2}\,,
        \end{split}
    \end{align*}
    where we used the fact that
    \[
    \frac{1}{\gamma}\left( 1-\frac{p-2}{\gamma p} \right)^{-1}=\frac{p}{n+2}\,.
    \]

Combining these estimates, we get
\begin{align*}
    \begin{split}
        \la^p
        &\le c (1+K)(1+\|  u\|_{L^\infty(\Om_T)}^p + \|H(z,|F|)\|_{L^\gamma(\Om_T)}^{\frac{p\gamma}{n+2}}  )\rho^{-p}\,.
    \end{split}
\end{align*}
    This completes the proof.
\end{proof}

\begin{remark}
    The condition $H(z,|F|)\in L^\gamma(\Om_T)$ for $\gamma=\tfrac{n+p}{p}$ is used only for the lemma above. 
    Since the integrability on $|F|$ is connected with the integrability of $u$, it seems that the term $H(z,|F|)\in L^\gamma(\Om_T)$ appears as we display the formulation of the Caccioppoli inequality regarding $\|u\|_{L_{\loc}^\infty(\Om_T)}$. Indeed, it is known that local $L^\infty$ estimate of $u$ holds for the heat equation when the source term $|F|\in L^{I}$ with $I> n+2$.
\end{remark}

As in the proof of previous lemma, the main idea of the $p$-intrinsic case is to reduce the right hand side of the Caccioppoli type inequality to the the right hand side of the Caccioppoli type inequality of the $p$-Laplace equation.
In the next lemma, we recover the phase criterion in \cite{KKM}.
\begin{lemma}\label{p_intr}
    Suppose \ref{p1} and \ref{p2}. Let  $\cc$ be given in Lemma~\ref{lemma_decay}. Then  
    \[
    \sup_{z\in Q_{4\rho}^\la(z_0)} a(z)\le \cc^{q}\|u\|_{L^\infty(\Om_T)}^{p-q}\la^{p-q}\,.
    \] 
\end{lemma}

\begin{proof}
Using \ref{p1} and Lemma~\ref{lemma_decay}, we get 
    \[
    \sup_{z\in Q_{4\rho}^\la(z_0)} a(z)\le K \|u\|^{p-q}_{L^\infty(\Om_T)}\rho^{q-p}\le  K\|u\|_{L^\infty(\Om_T)}^{p-q}\cc^{q-p}\la^{p-q}\,,
    \]
The the proof is completed by using the fact that $\cc>K^\frac{1}{p}$.
\end{proof}

To estimate the right hand side of the Caccioppoli inequality, we further estimate the second term on the right hand side of Lemma~\ref{poincare_lem} in the $p$-intrinsic geometry. The next lemma immediately follows from the previous lemma.

\begin{lemma}\label{p_poincare_pre}
	Suppose \ref{p1} and \ref{p2}.
    For $s\in[2\rho,4\rho]$, 
	there exists a constant $c=c(\data)$ such that
	\begin{align*}
		\begin{split}
			&\fiint_{Q_{s}^\la(z_0)}(|\na u|^{p-1}+a(z)|\na u|^{q-1}
            +|F|^{p-1}+a(z)|F|^{q-1}
            )\,dz\\
			&\le \fiint_{Q_{s}^\la(z_0)}
            (
            |\na u|
            +|F|)
            ^{p-1}\,dz 
			+c\la^{-1+\frac{p}{q}}\fiint_{Q_{s}^\la(z_0)}a(z)^\frac{q-1}{q}
            (
            |\na u|
            +|F|)
            ^{q-1}\,dz\,.
		\end{split}
	\end{align*}
\end{lemma}

Next, we provide a $p$-intrinsic parabolic Poincar\'e inequality.
\begin{lemma}\label{p_poincare}
Suppose \ref{p1} and \ref{p2}.
	For $s\in[2\rho,4\rho]$ and $\theta\in((q-1)/q,1]$, 
	there exists a constant $c=c(\data)$, such that 
	\begin{align*}
		\fiint_{Q_{s}^\la(z_0)}\frac{|u-u_{Q_{s}^\la(z_0)}|^{\theta p}}{s^{\theta p}}\,dz
			\le c\fiint_{Q_{s}^\la(z_0)} ( H(z,|\na u|)+ H(z,|F|)   )^\theta\,dz\,.
	\end{align*}
\end{lemma}

\begin{proof}
	By Lemma~\ref{poincare_lem} and Lemma~\ref{p_poincare_pre}, there exists a constant $c=c(\data)$ such that
	\begin{align*}
		\begin{split}
			&\fiint_{Q_{s}^\la(z_0)}\frac{|u-u_{Q_{s}^\la(z_0)}|^{\theta p}}{s^{\theta p}}\,dz\\
			&\le c\fiint_{Q_{s}^\la(z_0)}|\na u|^{\theta p}\,dz+c\left(\la^{2-p}\fiint_{Q_{s}^\la(z_0)} |\na u|^{p-1}+|F|^{p-1}\,dz\right)^{\theta p}\\
			&\qquad+c\left(\la^{1-p+\frac{p}{q}}\fiint_{Q_{s}^\la(z_0)}a(z)^\frac{q-1}{q} ( |\na u|^{q-1}+|F|^{q-1} )\,dz\right)^{\theta p}\,.
		\end{split}
	\end{align*}
    Applying H\"older's inequality to last two terms, we obtain
\begin{align*}
		\begin{split}
			&\fiint_{Q_{s}^\la(z_0)}\frac{|u-u_{Q_{s}^\la(z_0)}|^{\theta p}}{s^{\theta p}}\,dz\\
			&\le c\fiint_{Q_{s}^\la(z_0)}|\na u|^{\theta p}\,dz+c\la^{(2-p)\theta p}\left(\fiint_{Q_{s}^\la(z_0)} (|\na u|^{ p} +|F|^{ p})^\theta \,dz\right)^{p-1}\\
            &\qquad+c\la^{(1-p+\frac{p}{q})\theta p}\left(\fiint_{Q_{s}^\la(z_0)} (a(z) (|\na u|^{q} + |F|^{q}) )^\theta\,dz\right)^{\frac{(q-1) p}{q}}\,.
		\end{split}
	\end{align*}
    To proceed further, we use \ref{p4} to have
    \begin{align*}
        \begin{split}
            &\left(\fiint_{Q_{s}^\la(z_0)} (|\na u|^{ p} +|F|^{ p})^\theta \,dz\right)^{p-1}\\
            &= \left(\fiint_{Q_{s}^\la(z_0)} (|\na u|^{ p} +|F|^{ p})^\theta \,dz\right)^{p-2} \left(\fiint_{Q_{s}^\la(z_0)} (|\na u|^{ p} +|F|^{ p})^\theta \,dz\right) \\
            &\le \left(\fiint_{Q_{s}^\la(z_0)} (|\na u|^{ p} +|F|^{ p}) \,dz\right)^{(p-2)\theta} \left(\fiint_{Q_{s}^\la(z_0)} (|\na u|^{ p} +|F|^{ p})^\theta \,dz\right)\\
            &\le \la^{p(p-2)\theta}\left(\fiint_{Q_{s}^\la(z_0)} (|\na u|^{ p} +|F|^{ p})^\theta \,dz\right)\,.
        \end{split}
    \end{align*}
    Similarly, we have
    \begin{align*}
        \begin{split}
            &\left(\fiint_{Q_{s}^\la(z_0)} (a(z) (|\na u|^{q} + |F|^{q})^\theta\,dz\right)^{\frac{(q-1) p}{q}}\\
            &\le \left(\fiint_{Q_{s}^\la(z_0)} (a(z) (|\na u|^{q} + |F|^{q})^\theta\,dz\right)^{\frac{(q-1) p}{q} -1 } \left(\fiint_{Q_{s}^\la(z_0)} (a(z) (|\na u|^{q} + |F|^{q})^\theta\,dz\right) \\
            &\le \la^{ ( p-\frac{p}{q} -1)\theta p }\left(\fiint_{Q_{s}^\la(z_0)} (a(z) (|\na u|^{q} + |F|^{q})^\theta\,dz\right)\,.
        \end{split}
    \end{align*}
    Substituting these inequalities, the conclusion of lemma holds.
\end{proof}

With the aim of estimating $L^\infty$--$L^2$ term in the $p$-intrinsic cylinder, we denote
\[
S(u,Q_{s}^\la(z_0)):=\sup_{I_{s}^\la(t_0)}\fint_{B_{s}(x_0)}\frac{|u-u_{Q_{s}^\la(z_0)}|^2}{s^2}\,dx\,.
\] 
 
\begin{lemma}\label{p_sup_lem}
Suppose \ref{p1} and \ref{p2}.
	There exists a constant $c=c(\data)$ such that 
	\[
		S(u,Q_{2\rho}^\la(z_0))=\sup_{I^\la_{2\rho}(t_0)}\fint_{B_{2\rho}(x_0)}\frac{|u-u_{Q_{2\rho}^\la(z_0)}|^2}{\rho^2}\,dx\le c\la^2\,.
	\]	
\end{lemma}

\begin{proof}
	We take $2\rho\le \rho_1<\rho_2\le 4\rho$. Then we apply Lemma~\ref{caccio_lem} to have
	\begin{align*}
		\begin{split}
			&\la^{p-2}S(u,Q_{\rho_1}^\la(z_0))\\
			&\le \frac{c\rho_2^q}{(\rho_2-\rho_1)^q}\fiint_{Q_{\rho_2}^\la(z_0)}\left(\frac{|u-u_{Q_{\rho_2}^\la(z_0)}|^p}{\rho_2^p}+a(z)\frac{|u-u_{Q_{\rho_2}^\la(z_0)}|^q}{\rho_2^q}\right)\,dz\\
			&\qquad+\frac{c \rho_2^2}{(\rho_2-\rho_1)^2}\la^{p-2}\fiint_{Q_{\rho_2}^\la(z_0)}\frac{|u-u_{Q_{\rho_2}^\la(z_0)}|^2}{\rho_2^2}\,dz
            +c\fiint_{Q_{\rho_2}^\la(z_0)}H(z,|F|)\,dz\,,
		\end{split}
	\end{align*}
   where $c=c(n,p,q,\nu,L)$.
   To estimate the first integral, we use \ref{p1} to get
   \begin{align*}
       \begin{split}
           a(z)\frac{|u-u_{Q_{\rho_2}^\la(z_0)}|^q}{\rho_2^q}
           &= a(z)\frac{|u-u_{Q_{\rho_2}^\la(z_0)}|^{q-p}}{\rho_2^{q-p}}\frac{|u-u_{Q_{\rho_2}^\la(z_0)}|^p}{\rho_2^p}\\
           &\le c  \frac{|u-u_{Q_{\rho_2}^\la(z_0)}|^p}{\rho_2^p} \left( \sup_{z\in Q_{4\rho}^\la(z_0)} a(z) \right) \frac{\|u\|^{q-p}_{L^\infty(\Om_T)}}{\rho^{q-p}} \\
           &\le c\frac{|u-u_{Q_{\rho_2}^\la(z_0)}|^p}{\rho_2^p}\,,
       \end{split}
   \end{align*}
   where $c=c(\data)$. For the second integral, we use Poincar\'e inequality in the spatial direction to get
	\begin{align*}
		\begin{split}
			\fiint_{Q_{\rho_2}^\la(z_0)}\frac{|u-u_{Q_{\rho_2}^\la(z_0)}|^2}{\rho_2^2}\,dz &= \fint_{I_{\rho_2}^\la(t_0)} \fint_{B_{\rho_2}(x_0)}\frac{|u-u_{Q_{\rho_2}^\la(z_0)}|^2}{\rho_2^2}\,dx\,dt\\
			&\le c\fint_{I_{\rho_2}^\la(t_0)}\left(\fint_{B_{\rho_2}(x_0)} \Bigg(\frac{|u-u_{Q_{\rho_2}^\la(z_0)}|^p}{\rho_2^p}+|\na u|^p \Bigg)\,dx\right)^\frac{1}{p}\\
            &\qquad\times\left(\fint_{B_{\rho_2}(x_0)}\frac{|u-u_{Q_{\rho_2}^\la(z_0)}|^2}{\rho_2^2}\,dx\right)^\frac{1}{2}\,dt\\
            &\le c \left(\fiint_{Q_{\rho_2}^\la(z_0)} \left( \frac{|u-u_{Q_{\rho_2}^\la(z_0)}|^p}{\rho_2^p}+|\na u|^p\right)\,dz\right)^\frac{1}{p}\\
            &\qquad \times S(u,Q_{\rho_2}^\la(z_0)) ^\frac{1}{2}\,,
		\end{split}
	\end{align*}
 where $c=c(n,N,p)$.
   Therefore, it follows
   \begin{align*}
		\begin{split}
			\la^{p-2}S(u,Q_{\rho_1}^\la(z_0))
			&\le \frac{c\rho_2^q}{(\rho_2-\rho_1)^q}\fiint_{Q_{\rho_2}^\la(z_0)}\frac{|u-u_{Q_{\rho_2}^\la(z_0)}|^p}{\rho_2^p}\,dz\\
			&\qquad+\frac{c \rho_2^2}{(\rho_2-\rho_1)^2}\la^{p-2} S(u,Q_{\rho_2}^\la(z_0))^\frac{1}{2}\\
            &\qquad\qquad\times\left(\fiint_{Q_{\rho_2}^\la(z_0)} \Bigg(\frac{|u-u_{Q_{\rho_2}^\la(z_0)}|^p}{\rho_2^p}+|\na u|^p \Bigg)\,dz\right)^\frac{1}{p}\\
            &\qquad+c\fiint_{Q_{\rho_2}^\la(z_0)}H(z,|F|)\,dz\,.
		\end{split}
	\end{align*}
    Now applying \ref{p4} and Lemma~\ref{p_poincare}, we get
	\begin{align*}
		\begin{split}
			&S(u,Q_{\rho_1}^\la(z_0))\le c\frac{\rho_2^q}{(\rho_2-\rho_1)^q}\la^2+c\frac{ \rho_2^2}{(\rho_2-\rho_1)^2}\la\ S(u,Q_{\rho_2}^\la(z_0))^\frac{1}{2}\,.
		\end{split}
	\end{align*}
 Finally, we  apply Young's inequality to the last term
    \begin{align*}
    	\begin{split}
    		&S(u,Q_{\rho_1}^\la(z_0))\le \frac{1}{2}S(u,Q_{\rho_2}^\la(z_0))+ c\frac{\rho_2^{q+2}}{(\rho_2-\rho_1)^{q+2}}\la^2.
    	\end{split}
    \end{align*}
	The proof is concluded by an application of Lemma~\ref{tech_lem}.
\end{proof}

In order to proceed further with the proof of the reverse H\"older inequality, we divide of reasoning into steps starting with estimating the right hand side of the Caccioppoli inequality with the use of the Gagliargo--Nirenberg lemma. 

\begin{lemma}\label{p_reverse_lem_pre}
Suppose \ref{p1} and \ref{p2}.
	There exist constants $c=c(\data)$ and $\theta_0=\theta_0(n,p)\in(0,1)$, such that for any $\theta\in(\theta_0,1)$ we have
	\begin{align*}
		\begin{split}
			&\fiint_{Q_{2\rho}^\la(z_0)}\frac{|u-u_{Q_{2\rho}^\la(z_0)}|^p}{\rho^p}\,dz+\la^{p-2}\fiint_{Q_{2\rho}^\la(z_0)}\frac{|u-u_{Q_{2\rho}^\la(z_0)}|^2}{\rho^2}\,dz\\
			&\le  c\la^{p-1}\left(\fiint_{Q_{2\rho}^\la(z_0)}H(z,|\na u|)^\theta\,dz\right)^\frac{1}{\theta p}+c\la^{p-1}\left(\fiint_{Q_{2\rho}^\la(z_0)}H(z,|F|)\,dz\right)^\frac{1}{ p}\,.
		\end{split}
	\end{align*}
\end{lemma}

\begin{proof}
	We begin with the first term on the left hand side of the display. We apply Lemma~\ref{sobolev_lem} with  $\sig=p$, $s=\theta p$, $r=2$ and $\vartheta = \theta\in(n/(n+2),1)$. Then, we get
	\begin{align*}
 \begin{split}
			&\fiint_{Q_{2\rho}^\la(z_0)}\frac{|u-u_{Q_{2\rho}^\la(z_0)}|^p}{\rho^p}\,dz\\
			&\le c\fiint_{Q_{2\rho}^\la(z_0)}\left(\frac{|u-u_{Q_{2\rho}^\la(z_0)}|^{\theta p}}{\rho^{\theta p}}+|\na u|^{\theta p}\right)\,dz\left(S(u,Q_{2\rho}^\la(z_0))\right)^{\frac{(1-\theta)p}{2}}\,,
		\end{split}
	\end{align*}
 where $c=c(n,N,p)$. 
Now, we employ Lemma~\ref{p_poincare} and Lemma~\ref{p_sup_lem}. Then, we get
\begin{align*}
    \begin{split}
        \fiint_{Q_{2\rho}^\la(z_0)}\frac{|u-u_{Q_{2\rho}^\la(z_0)}|^p}{\rho^p}\,dz
         &\le c\la^{(1-\theta )p}\fiint_{Q_{2\rho}^\la(z_0)} (H(z,|\na u|) + H(z,|F|))^\theta\,dz\,.
    \end{split}
\end{align*}
Moreover, using \ref{p4} and H\"older's inequality, we get
\begin{align*}
    \begin{split}
        &\fiint_{Q_{2\rho}^\la(z_0)} (H(z,|\na u|) + H(z,|F|))^\theta\,dz\\
        &=\left( \fiint_{Q_{2\rho}^\la(z_0)} (H(z,|\na u|) + H(z,|F|))^\theta\,dz \right)^{1-\frac{1}{\theta p}} \\
        &\qquad\times \left( \fiint_{Q_{2\rho}^\la(z_0)} (H(z,|\na u|) + H(z,|F|))^\theta\,dz \right)^\frac{1}{\theta p}\\
        &\le \la^{p\theta-1}\left( \fiint_{Q_{2\rho}^\la(z_0)} (H(z,|\na u|) + H(z,|F|))^\theta\,dz \right)^\frac{1}{\theta p}\\
        &\le c\la^{p\theta-1}\left( \fiint_{Q_{2\rho}^\la(z_0)} H(z,|\na u|)^\theta\,dz \right)^\frac{1}{\theta p} +c\la^{p\theta-1}\left( \fiint_{Q_{2\rho}^\la(z_0)} H(z,| F|)\,dz \right)^\frac{1}{ p}\,,
    \end{split}
\end{align*}
for some constant $c=c(data)>0$.
It remains to show the remained term, we again apply Lemma~\ref{p_sup_lem} with $\sigma = 2, s= \theta p, \vartheta = \tfrac{1}{2}$ and $ r=2$ where $\theta\in(2n/((n+2)p),1)$. Then along with Lemma~\ref{p_poincare} and Lemma~\ref{p_sup_lem}
	\begin{align*}
		\begin{split}
			&\fiint_{Q_{2\rho}^\la(z_0)}\frac{|u-u_{Q_{2\rho}^\la(z_0)}|^2}{\rho^2}\,dz\\
			&\le c \fint_{I^\la_{2\rho}(t_0)}\left(\fint_{B_{2\rho}(x_0)}\left(\frac{|u-u_{Q_{2\rho}^\la(z_0)}|^{\theta  p}}{\rho^{\theta  p}}+|\na u|^{\theta  p}\right)\,dx\right)^\frac{1}{\theta p}\left(S(u,Q_{2\rho}^\la(z_0))\right)^\frac{1}{2}\,dt\\
			&\le c\la \left(\fiint_{Q_{2 \rho}^\la(z_0)} ( H(z,|\na u|) +H(z,|F|)
 )^\theta\,dz\right)^\frac{1}{\theta p}\\
 &\le c\la \left(\fiint_{Q_{2 \rho}^\la(z_0)}  H(z,|\na u|)^\theta\,dz\right)^\frac{1}{\theta p}+c\la \left(\fiint_{Q_{2 \rho}^\la(z_0)}  H(z,|F|)\,dz\right)^\frac{1}{ p}\,,
		\end{split}
	\end{align*}
 where $c=c(\data)$. This completes the proof.
 
\end{proof}

We state the reverse H\"older inequality.
\begin{lemma}\label{p_reverse_lem}
Suppose \ref{p1} and \ref{p2}.
	There exist constants $c=c(\data)$ and $\theta_0=\theta_0(n,p)\in(0,1)$ such that for any $\theta\in(\theta_0,1)$ there holds
	\begin{align*}
    \fiint_{Q_{\rho}^\la(z_0)}H(z,|\na u|)\,dz
			\le c\left(\fiint_{Q_{2\rho}^\la(z_0)}H(z,|\na u|)^\theta\,dz\right)^\frac{1}{\theta}+c\fiint_{Q_{2\rho}^\la(z_0)}H(z,|F|)\,dz\,.
	\end{align*}	
\end{lemma}
\begin{proof}
	By Lemma~\ref{caccio_lem}, we have
	\begin{align*}
		\begin{split}
            &\fiint_{Q_{\rho}^\la(z_0)}H(z,|\na u|)\,dz\\
			&\le c\fiint_{Q_{2\rho}^\la(z_0)}\left(\frac{|u-u_{Q_{2\rho}^\la(z_0)}|^p}{\rho^p}+a(z)\frac{|u-u_{Q_{2\rho}^\la(z_0)}|^q}{\rho^q}\right)\,dz\\
			&\qquad+c\la^{p-2}\fiint_{Q_{2\rho}^\la(z_0)}\frac{|u-u_{Q_{2\rho}^\la(z_0)}|^2}{\rho^2}\,dz+ c\fiint_{Q_{2\rho}^\la(z_0)}H(z,|F|)\,dz\,,
		\end{split}
	\end{align*}
	where $c=c(n,p,q,\nu,L)$. We apply \ref{p1} to get
    \begin{align*}
        \begin{split}
            &\fiint_{Q_{\rho}^\la(z_0)}H(z,|\na u|)\,dz\\
			&\le c\fiint_{Q_{2\rho}^\la(z_0)}\frac{|u-u_{Q_{2\rho}^\la(z_0)}|^p}{\rho^p}\,dz+c\la^{p-2}\fiint_{Q_{2\rho}^\la(z_0)}\frac{|u-u_{Q_{2\rho}^\la(z_0)}|^2}{\rho^2}\,dz\\
            &\qquad + c\fiint_{Q_{2\rho}^\la(z_0)}H(z,|F|)\,dz\,,
        \end{split}
    \end{align*}
    where $c=c(\data)$. Then applying Lemma~\ref{p_reverse_lem_pre} and \ref{p3}, it follows that
    \begin{align*}
        \begin{split}
            &\fiint_{Q_{\rho}^\la(z_0)}H(z,|\na u|)\,dz\\
			&\le c\la^{p-1}\left(\fiint_{Q_{2\rho}^\la(z_0)}H(z,|\na u|)^\theta\,dz\right)^\frac{1}{\theta p}+c\la^{p-1}\left(\fiint_{Q_{2\rho}^\la(z_0)}H(z,|F|)\,dz\right)^\frac{1}{p} \\
            &\qquad+ c\fiint_{Q_{2\rho}^\la(z_0)}H(z,|F|)\,dz\\
            &= c\left(\fiint_{Q_{\rho}^\la(z_0)}H(z,|\na u|)+ H(z,|F|) \,dz\right)^\frac{p-1}{p}\left(\fiint_{Q_{2\rho}^\la(z_0)}H(z,|\na u|)^\theta\,dz\right)^\frac{1}{\theta p}\\
            &\qquad+ c\left(\fiint_{Q_{\rho}^\la(z_0)}H(z,|\na u|)+ H(z,|F|) \,dz\right)^\frac{p-1}{p}\left(\fiint_{Q_{2\rho}^\la(z_0)}H(z,|F|)\,dz\right)^\frac{1}{ p}\\
            &\qquad +c\fiint_{Q_{2\rho}^\la(z_0)}H(z,|F|)\,dz\,.
        \end{split}
    \end{align*}
    The conclusion follows by Young's inequality.
\end{proof}

We estimate the right-hand side in the previous lemma further to apply of arguments from Gehring's lemma in the next section. Recall the upper level set $\Psi$ of $H(z,|\na u|)$ in \eqref{upper_level_set}. We denote the upper level set of $H(z,|F|)$ as
\begin{align}\label{upper_F}
    \Theta(\varkappa)=\{z\in \Om_T: H(z,|F(z)|)>\varkappa\}\,.
\end{align}

\begin{proposition}\label{p_est_vitali} 
Suppose \ref{p1} and \ref{p2}.
	There exist constants $c=c(\data)$ and $\theta_0=\theta_0(n,p)\in(0,1)$, such that for any $\theta\in(\theta_0,1)$ we have
	\begin{align*}
			\begin{split}
			    \iint_{Q_{2\cv\rho}^\la(z_0)}H(z,|\na u|)\,dz
			&\le c\La^{1-\theta}\iint_{Q_{2\rho}^\la(z_0)\cap \Psi(c^{-1}\La)}H(z,|\na u|)^\theta\,dz\\
            &\qquad+c\iint_{Q_{2\rho}^\la(z_0)\cap \Theta(c^{-1}\La)}H(z,|F|) \,dz\,.
			\end{split}
	\end{align*}
    Here, $\kappa$ is defined in \eqref{K_and_kappa} and it appears in \ref{p4}.
 \end{proposition}
\begin{proof}
	To begin with, we estimate the right hand side of the display in Lemma~\ref{p_reverse_lem}. Since \ref{p1} holds, we have
	\begin{align*}
	    \begin{split}
    \la^p=\fiint_{Q_{\rho}^\la(z_0)} (H(z,|\na u|)+ H(z,|F|)  )\,dz
   &\le c\la^{p(1-\theta)}\fiint_{Q_{2\rho}^\la(z_0)}H(z,|\na u|)^\theta \,dz\\
   &\qquad + c\fiint_{Q_{2\rho}^\la(z_0)}H(z,|F|) \,dz\,,
	    \end{split}
	\end{align*}
    where $c=c(\data)$. For this fixed constant $c$ above, we divide the domain of the integral $H(z,|\na u|)$ into $Q_{2\rho}^\la(z_0)\cap \Psi((4c)^{-1/\theta}\la^p)$ and its complements and similarly, $Q_{2\rho}^\la(z_0)\cap \Theta((4c)^{-1}\la^p)$ and its complement for $H(z,|F|)$. Then, there holds
     \begin{align*}
         \la^p
         &\le \frac{1}{2}\la^p 
         +c\frac{\la^{p(1-\theta)}}{|Q_{2\rho}^\la|}\iint_{Q_{2\rho}^\la(z_0)\cap \Psi((4c)^{-1/\theta}\la^p)}H(z,|\na u|)^{\theta }\,dz\\
         &\qquad +\frac{c}{|Q_{2\rho}^\la|}\iint_{Q_{2\rho}^\la(z_0)\cap \Theta((4c)^{-1}\la^p)} H(z,|F|) \,dz
     \end{align*}
    or, equivalently,
     \begin{align*}
         \begin{split}
             \la^p
         &\le 2c\frac{\la^{p(1-\theta)}}{|Q_{2\rho}^\la|}\iint_{Q_{2\rho}^\la(z_0)\cap \Psi((4c)^{-1/\theta}\la^p)}H(z,|\na u|)^{\theta }\,dz\\
         &\qquad + \frac{2c}{|Q_{2\rho}^\la|}\iint_{Q_{2\rho}^\la(z_0)\cap \Theta((4c)^{-1}\la^p)} H(z,|F|) \,dz
         \end{split}
     \end{align*}
     Now, recalling \eqref{K_and_kappa}, we replace the left hand side by using \ref{p4} in order to deduce
\begin{align*}
         \begin{split}
             \fiint_{Q_{2\cv\rho}^\la(z_0)}H(z,|\na u|)\,dz
         &\le 2c\frac{\la^{p(1-\theta)}}{|Q_{2\rho}^\la|}\iint_{Q_{2\rho}^\la(z_0)\cap \Psi((4c)^{-1/\theta}\la^p)}H(z,|\na u|)^{\theta }\,dz\\
         &\qquad+\frac{2c}{|Q_{2\rho}^\la|}\iint_{Q_{2\rho}^\la(z_0)\cap \Theta((4c)^{-1}\la^p)} H(z,|F|) \,dz\,.
         \end{split}
     \end{align*}
     Thus, we get
     \begin{align*}
         \begin{split}
             \iint_{Q_{2\cv\rho}^\la(z_0)}H(z,|\na u|)\,dz
         &\le 2\kappa^{n+2}c\la^{p(1-\theta)}\iint_{Q_{2\rho}^\la(z_0)\cap \Psi((4c)^{-1/\theta}\la^p)}H(z,|\na u|)^{\theta }\,dz\\
         &\qquad + 2\kappa^{n+2}c\iint_{Q_{2\rho}^\la(z_0)\cap \Theta((4c)^{-1}\la^p)} H(z,|F|) \,dz\,.
         \end{split}
     \end{align*}
    On the other side, we use Lemma~\ref{p_intr} to have
    \[
    \La =\la^p+a(z_0)\la^q\le (C+1)\la^p
    \]
    for some constant $C=C(\data)$. It follows that
\begin{align*}
    \begin{split}
         &    \iint_{Q_{2\cv\rho}^\la(z_0)}H(z,|\na u|)\,dz\\
         &\le 2\kappa^{n+2}c\la^{p(1-\theta)}\iint_{Q_{2\rho}^\la(z_0)\cap \Psi((4c)^{-1/\theta}(C+1)^{-1}\La)}H(z,|\na u|)^{\theta }\,dz\\
         &\qquad + 2\kappa^{n+2}c\iint_{Q_{2\rho}^\la(z_0)\cap \Theta((4c)^{-1}(C+1)^{-1}\La)} H(z,|F|) \,dz\,.
    \end{split}
\end{align*}
     Since $(4c)^{-1}(C+1)^{-1}> (4c)^{-1/\theta}(C+1)^{-1}$ holds,
     the proof is completed by replacing the constant above with $(4\kappa^{n+2}c)^\frac{1}{\theta_0}(C+1)$.
\end{proof}

\subsection{The \texorpdfstring{$(p,q)$}{(p,q)}-intrinsic case}  We concentrate on the proof of the reverse H\"older inequality in this regime provided in Lemma~\ref{q_reverse_lem} and its consequence ready to apply in the main proof, namely Proposition~\ref{q_est_vitali}. With the assumption \ref{q1}, we rewrite Lemma~\ref{poincare_lem} in the $(p,q)$-intrinsic geometry as follows.
\begin{lemma}\label{q_poincare_lem_pre}
Suppose \ref{q1} and \ref{q2}.
	For every $s\in[2\rho,4\rho]$, $m\in(1,q]$ and $\theta\in(1/m,1]$, there exists a constant $c=c(n,N,m,L,\ca)$, such that
	\begin{align*}
    		\begin{split}
			&\fiint_{G_s^\la(z_0)}\frac{|u-u_{G_s^\la(z_0)}|^{\theta m}}{s^{\theta m}}\,dz\\
            &\le c\fiint_{G_s^\la(z_0)}|\na u|^{\theta m}\,dz+c\left(\frac{\la^2}{\La}\fiint_{G_s^\la(z_0)}(|\na u|^{p-1}+a(z_0)|\na u|^{q-1})
   \,dz\right)^{\theta m}\\
   &\qquad+ c\left(\frac{\la^2}{\La}\fiint_{G_s^\la(z_0)}(|F|^{p-1}+a(z_0)|F|^{q-1})
   \,dz\right)^{\theta m}\,.
		\end{split}
	\end{align*}
\end{lemma}

We begin to estimate each term on the right hand side of Caccioppoli type inequality.
\begin{lemma}\label{q_poincare_1}
Suppose \ref{q1} and \ref{q2}.
    For $s\in[2\rho,4\rho]$ and $\theta\in((q-1)/q,1]$, 
	there exists a constant $c=c(n,N,p,L,\ca)$, such that 
	\begin{align*}
		\fiint_{G_{s}^\la(z_0)} \frac{|u-u_{G_{s}^\la(z_0)}|^{\theta p}}{s^{\theta p}}\,dz
			\le c\fiint_{{G^\la_s(z_0)}} (H(z_0,|\na u|+ H(z_0,|F|)  )^\theta\,dz\,.
	\end{align*}
\end{lemma}

\begin{proof}
    We employ Lemma~\ref{q_poincare_lem_pre} to have
    \begin{align*}
        \begin{split}
            &\fiint_{G_{s}^\la(z_0)} \frac{|u-u_{G_{s}^\la(z_0)}|^{\theta p}}{s^{\theta p}}\,dz\\
            &\le c\fiint_{G_s^\la(z_0)}|\na u|^{\theta p}\,dz+c\left(\frac{\la^2}{\La}\fiint_{G_s^\la(z_0)}(|\na u|^{p-1}+a(z_0)|\na u|^{q-1})
   \,dz\right)^{\theta p}\\
   &\qquad+c\left(\frac{\la^2}{\La}\fiint_{G_s^\la(z_0)}(|F|^{p-1}+a(z_0)|F|^{q-1})
   \,dz\right)^{\theta p}\,.
        \end{split}
    \end{align*}
    As $|\na u|^{\theta p}\le H(z_0,|\na u|)^{\theta}$ holds, it remains to estimate the last two terms. Since $p-1<\theta p$ and $p\ge 2$ hold, we apply H\"older's inequality and \ref{q5} to get
    \begin{align*}
        \begin{split}
            \left(\frac{\la^2}{\La}\fiint_{G_s^\la(z_0)}|\na u|^{p-1}\,dz\right)^{\theta p}
            &\le \left(\frac{\la^2}{\La}\right)^{\theta p} \left( \fiint_{G_s^\la(z_0)}|\na u|^{\theta p}\,dz\right)^{p-1}\\
            &\le \left(\frac{\la^2}{\La}\right)^{\theta p}\La^{\theta(p-2)} \fiint_{G_s^\la(z_0)} |\na u|^{\theta p}\,dz\\
            &\le \fiint_{G_s^\la(z_0)} |\na u|^{\theta p}\,dz\,,
        \end{split}
    \end{align*}
    where to obtain the last inequality, we used $\la \le\La^\frac{1}{p}$ so that
    \[
    \left(\frac{\la^2}{\La}\right)^{\theta p}\La^{\theta(p-2)}\le \La^{(\frac{2}{p}-1)\theta p}\La^{\theta(p-2)}=1\,.
    \]
    
    Similarly, for the remaining term, we again apply H\"older's inequality and \ref{q5}, and use the facts that $q-1<\theta q$ and $a(z_0)^\frac{1}{q}\la<\La^\frac{1}{q}$. Then, we get
    \begin{align*}
        \begin{split}
            \left(\frac{\la^2}{\La}\fiint_{G_s^\la(z_0)}a(z_0)|\na u|^{q-1}\,dz\right)^{\theta p}
            &= \left(\frac{(a(z_0)^\frac{1}{q}\la)\la}{\La}\fiint_{G_s^\la(z_0)}a(z_0)^\frac{q-1}{q}|\na u|^{q-1}\,dz\right)^{\theta p}\\
            &\le \left(\frac{\La^{\frac{1}{q}+\frac{1}{p}}}{\La}\right)^{\theta p}\left(\fiint_{G_s^\la(z_0)}a(z_0)^{\theta}|\na u|^{\theta q}\,dz\right)^\frac{(q-1)p}{q}\\
            &\le \left(\frac{\La^{\frac{1}{q}+\frac{1}{p}}}{\La}\right)^{\theta p}\La^{\frac{\theta (q-1)p}{q}-\theta}\fiint_{G_s^\la(z_0)}a(z_0)^{\theta}|\na u|^{\theta q}\,dz\\
            &\le \fiint_{G_s^\la(z_0)}H(z_0,|\na u|)^\theta\,dz\,,
        \end{split}
    \end{align*}
    where we used facts that $2\le p<q$ and $\tfrac{(q-1)p}{q}>1$. The same argument holds by replacing $|\na u|$ with $|F|$.
Combining these estimate, the proof is completed.
\end{proof}

Despite the previous lemma plays an important role as in the $p$-intrinsic case, the right hand side in Lemma~\ref{q_poincare_1} would be estimated by $\La^\theta$ rather than $\la^{\theta p}$, which is not applicable to obtain quantitative estimate of $L^\infty-L^2$. For this reason, we use the following estimate.
\begin{lemma}\label{q_trick}
Suppose \ref{q1} and \ref{q2}.
    For $s\in[2\rho,4\rho]$, there holds
    \[
    \fiint_{G^\la_s(z_0)} (|\na u|^p + |F|^p) \,dz\le (1+2^{1+\frac{q}{p}}\ca)^\frac{p}{q}\la^p\,.
    \]
\end{lemma}
\begin{proof}
    Note that it follows from H\"older's inequality, \ref{q1} and \ref{q5} that
    \begin{align*}
        &\fiint_{G^\la_s(z_0)} (|\na u|^p+|F|^p  )\,dz+\frac{a(z_0)}{2^{1+\frac{q}{p}}\ca}\left(\fiint_{G^\la_s(z_0)} ( |\na u|^p+|F|^p ) \,dz\right)^\frac{q}{p}\\
        &\le\fiint_{G^\la_s(z_0)} (|\na u|^p+|F|^p  )\,dz+\frac{a(z_0)}{2\ca}\fiint_{G^\la_s(z_0)} (|\na u|^q+|F|^q )\,dz\\
        &\le \fiint_{G^\la_s(z_0)} (|\na u|^p+|F|^p  )\,dz+\fiint_{G^\la_s(z_0)} a(z) (|\na u|^q+|F|^q )\,dz\\
        &<\La=\la^p+a(z_0)\la^q\,.
    \end{align*}
    Therefore, it cannot be true that
    \[
    \fiint_{G^\la_s(z_0)} (|\na u|^p+|F|^p  )\,dz> (1+2^{1+\frac{q}{p}}\ca)^\frac{p}{q}\la^p\,.
    \]
    Hence, the conclusion holds.
\end{proof}

\begin{lemma}\label{q_poincare_11}
Suppose \ref{q1} and \ref{q2}.
    For $s\in[2\rho,4\rho]$, 
	there exists a constant $c=c(n,N,p,L,\ca)$, such that 
	\begin{align*}
		\fiint_{G_{s}^\la(z_0)} \frac{|u-u_{G_{s}^\la(z_0)}|^{ p}}{s^{p}}\,dz
			\le c\la^p\,.
	\end{align*}
\end{lemma}

\begin{proof}
     We again employ Lemma~\ref{q_poincare_lem_pre} to have
    \begin{align*}
        \begin{split}
            &\fiint_{G_{s}^\la(z_0)} \frac{|u-u_{G_{s}^\la(z_0)}|^{ p}}{s^{ p}}\,dz\\
            &\le c\fiint_{G_s^\la(z_0)}|\na u|^{p}\,dz+c\left(\frac{\la^2}{\La}\fiint_{G_s^\la(z_0)}(|\na u|^{p-1}+a(z_0)|\na u|^{q-1})
   \,dz\right)^{ p}\\
   &\qquad + c\left(\frac{\la^2}{\La}\fiint_{G_s^\la(z_0)}(|F|^{p-1}+a(z_0)|F|^{q-1})
   \,dz\right)^{ p}\,.
        \end{split}
    \end{align*}
    Recalling $\La=\la^p+a(z_0)\la^q$ and $q-1\le p$, we use H\"older's inequality and Lemma~\ref{q_trick} to have
    \begin{align*}
        \begin{split}
            \fiint_{G_{s}^\la(z_0)} \frac{|u-u_{G_{s}^\la(z_0)}|^{ p}}{s^{ p}}\,dz
            \le c\la^p+c\left(\frac{\la^2}{\La} (\la^{p-1}+a(z_0)\la^{q-1})\right)^{ p}\le c\la^p\,.
        \end{split}
    \end{align*}
    The proof is completed.
\end{proof}

We are now going back to the $q$-Laplacian part version of Lemma~\ref{q_poincare_1}.

\begin{lemma}\label{q_poincare_2}
Suppose \ref{q1} and \ref{q2}.
    For $s\in[2\rho,4\rho]$ and $\theta\in((q-1)/q,1]$, 
	there exists a constant $c=c(n,N,q,L,\ca)$, such that 
	\begin{align*}
		a(z_0)^\theta\fiint_{G_{s}^\la(z_0)} \frac{|u-u_{G_{s}^\la(z_0)}|^{\theta q}}{s^{\theta q}}\,dz
			\le c\fiint_{G^\la_s(z_0)} (H(z_0,|\na u|) +H(z_0,|F|) )^\theta\,dz\,.
	\end{align*}
\end{lemma}

\begin{proof}
    Applying Lemma~\ref{q_poincare_lem_pre}, we obtain
        \begin{align*}
        \begin{split}
            &a(z_0)^\theta\fiint_{G_{s}^\la(z_0)} \frac{|u-u_{G_{s}^\la(z_0)}|^{\theta q}}{s^{\theta q}}\,dz\\
            &\le c\fiint_{G_s^\la(z_0)} H(z_0,|\na u|)^{\theta}\,dz+ca(z_0)^{\theta}\left(\frac{\la^2}{\La}\fiint_{G_s^\la(z_0)}(|\na u|^{p-1}+a(z_0)|\na u|^{q-1})
   \,dz\right)^{\theta q}\\
   &\qquad + ca(z_0)^{\theta}\left(\frac{\la^2}{\La}\fiint_{G_s^\la(z_0)}(|F|^{p-1}+a(z_0)|F|^{q-1})
   \,dz\right)^{\theta q}\,,
        \end{split}
    \end{align*}
    where we used $a(z_0)^{\theta}|\na u|^{\theta q}\le H(z_0,|\na u|)^{\theta}$. To estimate the remaining terms, we observe that
    \begin{align*}
        &\frac{\la^2}{\La}\fiint_{G_s^\la(z_0)}(|\na u|^{p-1}+a(z_0)|\na u|^{q-1})
   \,dz\\
   &=\frac{\la^2}{\la^{p}+a(z_0)\la^{q}}\fiint_{G_s^\la(z_0)}(|\na u|^{p-1}+a(z_0)|\na u|^{q-1})
   \,dz\\
   &\le \la^{2-p}\fiint_{G_s^\la(z_0)}|\na u|^{p-1}\,dz+\la^{2-q}\fiint_{G_s^\la(z_0)}|\na u|^{q-1}
   \,dz\,.
    \end{align*}
    Therefore, using  H\"older's inequality and the facts that $q-1\le p$ and Lemma~\ref{q_trick} so that
    \begin{align*}
        \begin{split}
            \fiint_{G_s^\la(z_0)}|\na u|^{q-1} \,dz
            &\le \left(\fiint_{G_s^\la(z_0)}|\na u|^{p} \,dz \right)^\frac{(q-2)}{p} \left(\fiint_{G_s^\la(z_0)}|\na u|^{q-1} \,dz\right)^\frac{1}{q-1}\\
            &\le \la^{q-2}\left(\fiint_{G_s^\la(z_0)}|\na u|^{q-1} \,dz\right)^\frac{1}{q-1}\,.
        \end{split}
    \end{align*}
    We obtain
    \begin{align*}
            \frac{\la^2}{\La}\fiint_{G_s^\la(z_0)}(|\na u|^{p-1}+a(z_0)|\na u|^{q-1})
   \,dz\le 2 \left(\fiint_{G_s^\la(z_0)}|\na u|^{q-1}\,dz\right)^\frac{1}{q-1}\,.
    \end{align*}
    Since we set $q-1<\theta q$, it follows by Lemma~\ref{q_trick} that
    \begin{align*}
        \begin{split}
            &\left(\frac{\la^2}{\La}\fiint_{G_s^\la(z_0)}(|\na u|^{p-1}+a(z_0)|\na u|^{q-1})
   \,dz\right)^{\theta q}\\
    &\le \left(2 \left(\fiint_{G_s^\la(z_0)}|\na u|^{q-1}\,dz\right)^\frac{1}{q-1}\right)^{\theta q}\,.
        \end{split}
    \end{align*}
    Hence, we obtain
    \begin{align*}
        \begin{split}
            a(z_0)^{\theta}\left(\frac{\la^2}{\La}\fiint_{G_s^\la(z_0)}(|\na u|^{p-1}+a(z_0)|\na u|^{q-1})
   \,dz\right)^{\theta q} &\le 2^q a(z_0)^\theta\fiint_{G_s^\la(z_0)}|\na u|^{\theta q}\,dz\\
   &\le 2^q \fiint_{G_s^\la(z_0)} H(z_0,|\na u|)^{\theta}\,dz\,.
        \end{split}
    \end{align*}
    The same argument holds by replacing $|\na u|$ with $|F|$.
    The proof is completed.
\end{proof}

As in the $p$-intrinsic case, we denote
\[
S(u,G_{s}^\la(z_0))=\sup_{J_{s}^\la(t_0)}\fint_{B_{s}(x_0)}\frac{|u-u_{G_{s}^\la(z_0)}|^2}{s^2}\,dx\,.
\]
For the reverse H\"older inequality, we again prove the following estimate.
\begin{lemma}\label{q_sup_lem}
Suppose \ref{q1} and \ref{q2}.
	There exists a constant $c=c(n,N,p,q,\nu,L,\ca)$ such that 
	\[
		S(u,G_{s}^\la(z_0))=\sup_{J^\la_{2\rho}(t_0)}\fint_{B_{2\rho}(x_0)}\frac{|u-u_{G_{2\rho}^\la(z_0)}|^2}{\rho^2}\,dx\le c\la^2.
	\]	
\end{lemma}

\begin{proof}
The proof is analogous to the proof in Lemma~\ref{p_sup_lem}.
	Let $2\rho\le \rho_1<\rho_2\le 4\rho$. We apply Lemma~\ref{caccio_lem} and \ref{q1} to have
	\begin{align}\label{q_sup_est}
		\begin{split}
			&\frac{\La}{\la^2}S(u,G_{\rho_1}^\la(z_0))\\
			&\le \frac{c\rho_2^q}{(\rho_2-\rho_1)^q}\fiint_{G_{\rho_2}^\la(z_0)}\left(\frac{|u-u_{G_{\rho_2}^\la(z_0)}|^p}{\rho_2^p}+a(z_0)\frac{|u-u_{G_{\rho_2}^\la(z_0)}|^q}{\rho_2^q}\right)\,dz\\
			&\qquad+\frac{c \rho_2^2}{(\rho_2-\rho_1)^2}\frac{\La}{\la^2}\fiint_{G_{\rho_2}^\la(z_0)}\frac{|u-u_{G_{\rho_2}^\la(z_0)}|^2}{\rho_2^2}\,dz+c\fiint_{G_{\rho_2}^\la(z_0)}  H(z,|F|) \,dz
		\end{split}
	\end{align}
   where $c=c(n,p,q,\nu,L,\ca)$. 
    To estimate the first term on the right hand side, we apply Lemma~\ref{q_poincare_1} and Lemma~\ref{q_poincare_2} along with \ref{q1} and \ref{q5}. Then
    \begin{align*}
        \begin{split}
            &\fiint_{G_{\rho_2}^\la(z_0)}\left(\frac{|u-u_{G_{\rho_2}^\la(z_0)}|^p}{\rho_2^p}+a(z_0)\frac{|u-u_{G_{\rho_2}^\la(z_0)}|^q}{\rho_2^q}\right)\,dz\\
            &\le c \fiint_{G_{\rho_2}^\la(z_0)} ( H(z_0,|\na u|)+ H(z_0,|F|) )\,dz\\
            &\le c \fiint_{G_{\rho_2}^\la(z_0)} ( H(z,|\na u|)+ H(z,|F|) )\,dz\\
            &\le c\La\,.
        \end{split}
    \end{align*}
    
    On the other side, for the second term, we use Poincar\'e inequality in the spatial direction, Lemma~\ref{q_trick} and Lemma~\ref{q_poincare_11} as in the proof of Lemma~\ref{p_sup_lem}. Then,
	\begin{align*}
		\begin{split}
			\fiint_{G_{\rho_2}^\la(z_0)}\frac{|u-u_{G_{\rho_2}^\la(z_0)}|^2}{\rho_2^2}\,dz
			&\le c\fint_{J_{\rho_2}^\la(t_0)}\left(\fint_{B_{\rho_2}(x_0)} \Bigg( \frac{|u-u_{G_{\rho_2}^\la(z_0)}|^p}{\rho_2^p}+|\na u|^p \Bigg) \,dx\right)^\frac{1}{p}\\
            &\qquad\times\left(\fint_{B_{\rho_2}(x_0)}\frac{|u-u_{G_{\rho_2}^\la(z_0)}|^2}{\rho_2^2}\,dx\right)^\frac{1}{2}\,dt\\
			&\le c\la \left(S(u,G_{\rho_2}^\la(z_0))\right)^\frac{1}{2}\,,
		\end{split}
	\end{align*}
 where $c=c(n,N,p,\ca)$.	Finally, \ref{q5} gives
 \[
 \fiint_{G_{\rho_2}^\la(z_0)}  H(z,|F|) \,dz\le \La\,.
 \] 
 Dividing \eqref{q_sup_est} by $\tfrac{\La}{\la^2}$, we get
	\begin{align*}
		\begin{split}
			&S(u,G_{\rho_1}^\la(z_0))\le c\frac{\rho_2^q}{(\rho_2-\rho_1)^q}\la^2+c\frac{ \rho_2^2}{(\rho_2-\rho_1)^2}\la\ S(u,G_{\rho_2}^\la(z_0))^\frac{1}{2}\,.
		\end{split}
	\end{align*}
 Moreover, Young's inequality gives
    \begin{align*}
    	\begin{split}
    		&S(u,G_{\rho_1}^\la(z_0))\le \frac{1}{2}S(u,G_{\rho_2}^\la(z_0))+ c\frac{\rho_2^{q+2}}{(\rho_2-\rho_1)^{q+2}}\la^2\,.
    	\end{split}
    \end{align*}
	The conclusion follows by Lemma~\ref{tech_lem}.
\end{proof}

In order to obtain the reverse H\"older inequality, we estimate each term on the right hand side of the Caccioppoli inequality.
\begin{lemma}\label{q_reverse_lem_pre}
Suppose \ref{q1} and \ref{q2}.
	There exist constants $c=c(n,N,p,q,\nu,L,\ca)$ and $\theta_0=\theta_0(n,p,q)\in(0,1)$, such that for any $\theta\in(\theta_0,1)$ we have
	\begin{align*}
		\begin{split}
			&\fiint_{G_{2\rho}^\la(z_0)} \left(\frac{|u-u_{G_{2\rho}^\la(z_0)}|^p}{\rho^p}+a(z_0)\frac{|u-u_{G_{2\rho}^\la(z_0)}|^q}{\rho^q}  \right)\,dz\\
			&\le  c\La^{1-\theta}\fiint_{G_{2\rho}^\la(z_0)}   (H(z_0,|\na u|)+H(z_0,|F|)  )^\theta\,dz\,.
		\end{split}
	\end{align*}
\end{lemma}

\begin{proof}
	To estimate the $p$-exponent part in the first term, we apply Lemma~\ref{sobolev_lem} with  $\sig=p$, $s=\theta p$, $r=2$ and $\vartheta = \theta\in(n/(n+2),1)$. Then, 
	\begin{align*}
 \begin{split}
			&\fiint_{G_{2\rho}^\la(z_0)}\frac{|u-u_{G_{2\rho}^\la(z_0)}|^p}{\rho^p}\,dz\\
			&\le c\fiint_{G_{2\rho}^\la(z_0)}\left(\frac{|u-u_{G_{2\rho}^\la(z_0)}|^{\theta p}}{\rho^{\theta p}}+|\na u|^{\theta p}\right)\,dz\left(S(u,G_{2\rho}^\la(z_0))\right)^{\frac{(1-\theta)p}{2}}\,,
		\end{split}
	\end{align*}
 where $c=c(n,N,p)$. 
 Similarly, replacing $p$ by $q$, there holds
 	\begin{align*}
 \begin{split}
			&a(z_0)\fiint_{G_{2\rho}^\la(z_0)}\frac{|u-u_{G_{2\rho}^\la(z_0)}|^q}{\rho^q}\,dz\\
			&\le c\fiint_{G_{2\rho}^\la(z_0)}a(z_0)^{\theta}\left(\frac{|u-u_{G_{2\rho}^\la(z_0)}|^{\theta q}}{\rho^{\theta q}}+|\na u|^{\theta q}\right)dz\ \left(a(z_0)^{1-\theta} S(u,G_{2\rho}^\la(z_0))^{\frac{(1-\theta)q}{2}}\right)\,.
		\end{split}
	\end{align*}
 Employing Lemma~\ref{q_poincare_1}, Lemma~\ref{q_poincare_2} and Lemma~\ref{q_sup_lem}, we get
\begin{align*}
    \begin{split}
        &\fiint_{G_{2\rho}^\la(z_0)} \left(\frac{|u-u_{G_{2\rho}^\la(z_0)}|^p}{\rho^p}+a(z_0)\frac{|u-u_{G_{2\rho}^\la(z_0)}|^q}{\rho^q}  \right)\,dz\\
         &\le c\fiint_{G_{2\rho}^\la(z_0)}  (H(z_0,|\na u|) +H(z_0,|F|)  )^\theta\,dz (H(z_0,\la))^{1-\theta}\\
         &= c\La^{1-\theta} \fiint_{G_{2\rho}^\la(z_0)} ( H(z_0,|\na u|) +H(z_0,|F|) )^\theta\,dz\,.
    \end{split}
\end{align*} 
This completes the proof.
 
\end{proof}

\begin{lemma}\label{q_reverse_lem_pre_2}
Suppose \ref{q1} and \ref{q2}.
There exist constants $c=c(n,N,p,q,\nu,L,\ca)$ and $\theta_0=\theta_0(n,p,q)\in(0,1)$, such that for any $\theta\in(\theta_0,1)$ we have
    \begin{align*}
        \begin{split}
            &\frac{\La}{\la^2}\fiint_{G_{2\rho}^\la(z_0)}\frac{|u-u_{G_{2\rho}^\la(z_0)}|^2}{\rho^2}\,dz\\
            &\le c\la^{p-1} \left(\fiint_{G_{2\rho}^\la(z_0)}|\na u|^{\theta p}\,dz\right)^\frac{1}{\theta p}+ca(z_0)^\frac{q-1}{q}\la^{q-1} \left(\fiint_{G_{2\rho}^\la(z_0)} a(z_0)^\theta|\na u|^{\theta q}\,dz\right)^\frac{1}{\theta q}\\
            &\qquad+c\la \left(\fiint_{G_{2\rho}^\la(z_0)}|\na u|^{\theta p}\,dz\right)^\frac{p-1}{\theta p}+ca(z_0)^\frac{1}{q}\la \left(\fiint_{G_{2\rho}^\la(z_0)} a(z_0)^\theta|\na u|^{\theta q}\,dz\right)^\frac{q-1}{\theta q}\\
            &\qquad+c\la \left(\fiint_{G_{2\rho}^\la(z_0)}|F|^{ p}\,dz\right)^\frac{p-1}{ p}+ca(z_0)^\frac{1}{q}\la \left(\fiint_{G_{2\rho}^\la(z_0)} a(z_0) |F|^{ q}\,dz\right)^\frac{q-1}{ q}\,.
        \end{split}
    \end{align*}
\end{lemma}

\begin{proof}
     We again apply Lemma~\ref{q_sup_lem} and Lemma~\ref{sobolev_lem} with $\sigma = 2, s= \theta p, \vartheta = \tfrac{1}{2}$ and $ r=2$ where $\theta\in(2n/((n+2)p),1)$. Then, we have
	\begin{align*}
		\begin{split}
			&\fiint_{G_{2\rho}^\la(z_0)}\frac{|u-u_{G_{2\rho}^\la(z_0)}|^2}{\rho^2}\,dz\\
			&\le c \fint_{J^\la_{2\rho}(t_0)}\left(\fint_{B_{2\rho}(x_0)}\left(\frac{|u-u_{G_{2\rho}^\la(z_0)}|^{\theta  p}}{\rho^{\theta  p}}+|\na u|^{\theta  p}\right)\,dx\right)^\frac{1}{\theta p}\left(S(u,G_{2\rho}^\la(z_0))\right)^\frac{1}{2}\,dt\\
			&\le c\la \left(\fiint_{G_{2 \rho}^\la(z_0)}\left(\frac{|u-u_{G_{2\rho}^\la(z_0)}|^{\theta  p}}{\rho^{\theta  p}}+|\na u|^{\theta  p}\right)\,dz\right)^\frac{1}{\theta p}\,.
		\end{split}
	\end{align*}
    To proceed further, we apply Lemma~\ref{q_poincare_lem_pre} to obtain
    \begin{align}\label{est_4_19}
        \begin{split}
            &\frac{\La}{\la^2}\fiint_{G_{2\rho}^\la(z_0)}\frac{|u-u_{G_{2\rho}^\la(z_0)}|^2}{\rho^2}\,dz\\
            &\le c\frac{\La}{\la} \left(  \fiint_{G_{2 \rho}^\la(z_0)}|\na u|^{\theta  p}\,dz  \right)^\frac{1}{\theta p}+c\la\fiint_{G_{2\rho}^\la(z_0)}(|\na u|^{p-1}+a(z_0)|\na u|^{q-1})\,dz\\
            &\qquad +c\la\fiint_{G_{2\rho}^\la(z_0)}(|F|^{p-1}+a(z_0)|F|^{q-1})\,dz\,.
        \end{split}
    \end{align}
    We estimate each term on the right hand side. For the first term, we observe
    \begin{align*}
        \begin{split}
            \frac{\La}{\la} \left(  \fiint_{G_{2 \rho}^\la(z_0)}|\na u|^{\theta  p}\,dz  \right)^\frac{1}{\theta p} &=\la^{p-1}   \left(  \fiint_{G_{2 \rho}^\la(z_0)}|\na u|^{\theta  p}\,dz  \right)^\frac{1}{\theta p}\\
            &\qquad+a(z_0)\la^{q-1}\left(  \fiint_{G_{2 \rho}^\la(z_0)}|\na u|^{\theta  p}\,dz  \right)^\frac{1}{\theta p}\,.
        \end{split}
    \end{align*}
    Leaving the first term on the right hand side of the above display, we estimate the second term. We apply H\"older's inequality to have
    \begin{align*}
        \begin{split}
            a(z_0)\la^{q-1}\left(  \fiint_{G_{2 \rho}^\la(z_0)}|\na u|^{\theta  p}\,dz  \right)^\frac{1}{\theta p}
            &\le a(z_0)\la^{q-1}\left(  \fiint_{G_{2 \rho}^\la(z_0)}|\na u|^{\theta  q}\,dz  \right)^\frac{1}{\theta q}\\
            &= a(z_0)^\frac{q-1}{q}\la^{q-1}\left( \fiint_{G_{2 \rho}^\la(z_0)} a(z_0)^{\theta}|\na u|^{\theta  q}\,dz  \right)^\frac{1}{\theta q}\,.
        \end{split}
    \end{align*}

    To estimate the second term in \eqref{est_4_19}, we use H\"older's inequality to get
    \begin{align*}
        \begin{split}
            &\la \fiint_{G_{2\rho}^\la(z_0)}(|\na u|^{p-1}+a(z_0)|\na u|^{q-1})\,dz\\
            &\le \la \left(\fiint_{G_{2\rho}^\la(z_0)}|\na u|^{\theta p}\,dz\right)^\frac{p-1}{\theta p}+a(z_0)^\frac{1}{q}\la \left(\fiint_{G_{2\rho}^\la(z_0)} a(z_0)^\theta|\na u|^{\theta q}\,dz\right)^\frac{q-1}{\theta q}\,.
        \end{split}
    \end{align*}
    Replacing $|\na u|$ by $|F|$ and applying H\"older's inequality, we get
    \begin{align*}
        \begin{split}
            &\la \fiint_{G_{2\rho}^\la(z_0)}(|F|^{p-1}+a(z_0)|F|^{q-1})\,dz\\
            &\le \la \left(\fiint_{G_{2\rho}^\la(z_0)}|F|^{p}\,dz\right)^\frac{p-1}{ p}+a(z_0)^\frac{1}{q}\la \left(\fiint_{G_{2\rho}^\la(z_0)} a(z_0)|F|^{ q}\,dz\right)^\frac{q-1}{ q}\,.
        \end{split}
    \end{align*}
    The proof is completed.
\end{proof}

\begin{lemma}\label{q_reverse_lem}
Suppose \ref{q1} and \ref{q2}.
	There exist constants $c=c(n,N,p,q,\nu,L,\ca)$ and $\theta_0=\theta_0(n,p,q)\in(0,1)$ such that for any $\theta\in(\theta_0,1)$ there holds
	\begin{align*}
    \fiint_{G_{\rho}^\la(z_0)}H(z,|\na u|)\,dz
			\le c\left(\fiint_{G_{2\rho}^\la(z_0)}H(z,|\na u|)^\theta\,dz\right)^\frac{1}{\theta}+c\fiint_{G_{2\rho}^\la(z_0)}H(z,|F|) \,dz\,.
	\end{align*}	
\end{lemma}

\begin{proof}
    Employing Lemma~\ref{caccio_lem} and \ref{q1}, it follows that
	\begin{align*}
		\begin{split}
            &\fiint_{G_{\rho}^\la(z_0)}H(z,|\na u|)\,dz\\
			&\le c\fiint_{G_{2\rho}^\la(z_0)}\left(\frac{|u-u_{G_{2\rho}^\la(z_0)}|^p}{\rho^p}+a(z_0)\frac{|u-u_{G_{2\rho}^\la(z_0)}|^q}{\rho^q}\right)\,dz\\
			&\qquad+c\frac{\La}{\la^2}\fiint_{G_{2\rho}^\la(z_0)}\frac{|u-u_{G_{2\rho}^\la(z_0)}|^2}{\rho^2}\,dz + c\fiint_{G_{2\rho}^\la(z_0)}H(z_0,|F|)\,dz\,,
		\end{split}
	\end{align*}
	where $c=c(n,p,q,\nu,L,\ca)$. For the right hand side, we apply Lemma~\ref{q_reverse_lem_pre} and Lemma~\ref{q_reverse_lem_pre_2}. Then, we obtain
    \begin{align*}
        \begin{split}
            &\fiint_{G_{\rho}^\la(z_0)}H(z,|\na u|)\,dz\\
            &\le c\La^{1-\theta}\fiint_{G_{2\rho}^\la(z_0)}   (H(z_0,|\na u|) +H(z_0,|F|)  )^\theta\,dz + c\fiint_{G_{2\rho}^\la(z_0)}   H(z_0,|F|)\,dz\\
            &\qquad+c\la \left(\fiint_{G_{2\rho}^\la(z_0)}|\na u|^{\theta p}\,dz\right)^\frac{p-1}{\theta p}+ca(z_0)^\frac{1}{q}\la \left(\fiint_{G_{2\rho}^\la(z_0)} a(z_0)^\theta|\na u|^{\theta q}\,dz\right)^\frac{q-1}{\theta q}\\
            &\qquad+c\la \left(\fiint_{G_{2\rho}^\la(z_0)}|F|^{ p}\,dz\right)^\frac{p-1}{ p}+ca(z_0)^\frac{1}{q}\la \left(\fiint_{G_{2\rho}^\la(z_0)} a(z_0) |F|^{ q}\,dz\right)^\frac{q-1}{ q}\\
            &\qquad +c\la^{p-1} \left(\fiint_{G_{2\rho}^\la(z_0)}|\na u|^{\theta p}\,dz\right)^\frac{1}{\theta p}\\
            &\qquad+ca(z_0)^\frac{q-1}{q}\la^{q-1} \left(\fiint_{G_{2\rho}^\la(z_0)} a(z_0)^\theta|\na u|^{\theta q}\,dz\right)^\frac{1}{\theta q}\,.
        \end{split}
    \end{align*}
    Using Young's inequality and $\La=\la^p+a(z_0)\la^q$, we deduce
    \begin{align*}
        \begin{split}
            \fiint_{G_{\rho}^\la(z_0)}H(z,|\na u|)\,dz
            &\le \frac{1}{2}\La+ c\left(\fiint_{G_{2\rho}^\la(z_0)}H(z_0,|\na u|)^\theta\,dz\right)^\frac{1}{\theta} \\
            &\qquad+ c\fiint_{G_{2\rho}^\la(z_0)}H(z_0,|F|)\,dz\,,
        \end{split}
    \end{align*}
    where $c=c(n,N,p,q,\nu,L,\ca)$. Finally applying \ref{q4} to absorb the first term on the right hand side to the left hand side and \ref{q1} to replace $H(z_0,|\na u|)\le 2\ca H(z,|\na u|)$ and $H(z_0,|F|)\le 2\ca H(z,|F|)$ in $G_{2\rho}^\la(z_0)$. The proof is completed.
\end{proof}

Recalling the upper level set of $H(z,|\na u|)$ and $H(z,|F|)$ in \eqref{upper_level_set} and \eqref{upper_F}, we close this section with the following fact.
\begin{proposition}
    \label{q_est_vitali}
Suppose \ref{q1} and \ref{q2}.
	There exist constants $\theta_0=\theta_0(n,p,q)\in(0,1)$ and $c=c(n,N,p,q,\nu,L,\ca)$  such that for any $\theta\in(\theta_0,1)$ we have
	\begin{align*}
			\begin{split}
			    \iint_{G_{2\cv\rho}^\la(z_0)}H(z,|\na u|)\,dz
			&\le c\La^{1-\theta}\iint_{G_{2\rho}^\la(z_0)\cap \Psi(c^{-1}\La)}H(z,|\na u|)^\theta\,dz \\
            &\qquad+c\iint_{G_{2\rho}^\la(z_0)\cap \Theta(c^{-1}\La)}H(z,|F|)\,dz\,.
			\end{split}
	\end{align*}
    Here, $\kappa$ is defined in \eqref{K_and_kappa} and it appears in \ref{q5},
\end{proposition}

\begin{proof}
    The proof is analogous to the proof of Proposition~\ref{p_est_vitali}. By Lemma~\ref{q_reverse_lem} and \ref{q4}, we have
	\begin{align*}
	    \begin{split}
	        \La=\fiint_{G_{\rho}^\la(z_0)}  H(z,|\na u|) + H(z,|F|) \,dz
   &\le c\La^{1-\theta}\fiint_{G_{2\rho}^\la(z_0)}H(z,|\na u|)^\theta \,dz\\
   &\qquad+ c\fiint_{G_{2\rho}^\la(z_0)} H(z,|F|)\,dz
	    \end{split}
	\end{align*}
    for $c=c(n,N,p,q,\nu,L,\ca)>0$. Denoting $c$ from the above display, we decompose the referenced domain of integrals. We have
     \begin{align*}
         \begin{split}
             \La
         &\le \frac{1}{2}\La 
         +c\frac{\La^{1-\theta}}{|G_{2\rho}^\la|}\iint_{G_{2\rho}^\la(z_0)\cap \Psi((4c)^{-1/\theta}\La)}H(z,|\na u|)^{\theta }\,dz\\
         &\qquad+\frac{c}{|G_{2\rho}^\la|}\iint_{G_{2\rho}^\la(z_0)\cap \Theta((4c)^{-1}\La)}H(z,|F|) \,dz
         \end{split}
     \end{align*}
    and thus, it follows that
     \begin{align*}
         \begin{split}
             \La
         &\le 2c\frac{\La^{1-\theta}}{|G_{2\rho}^\la|}\iint_{G_{2\rho}^\la(z_0)\cap \Psi((4c)^{-1/\theta}\La)}H(z,|\na u|)^{\theta }\,dz\\
         &\qquad+\frac{2c}{|G_{2\rho}^\la|}\iint_{G_{2\rho}^\la(z_0)\cap \Theta((4c)^{-1}\La)}H(z,|F|) \,dz\,.
         \end{split}
     \end{align*}
     We use \ref{q5} to replace the left hand side as follows
\begin{align*}
         \begin{split}
             \fiint_{G_{2\cv\rho}^\la(z_0)}H(z,|\na u|)\,dz
         &\le 2c\frac{\La^{1-\theta}}{|G_{2\rho}^\la|}\iint_{G_{2\rho}^\la(z_0)\cap \Psi((4c)^{-1/\theta}\La)}H(z,|\na u|)^{\theta }\,dz\\
         &\qquad+\frac{2c}{|G_{2\rho}^\la|}\iint_{G_{2\rho}^\la(z_0)\cap \Theta((4c)^{-1}\La)}H(z,|F|) \,dz\,.
         \end{split}
     \end{align*}
     Thus, we get
     \begin{align*}
         \begin{split}
             \iint_{G_{2\cv\rho}^\la(z_0)}H(z,|\na u|)\,dz
         &\le 2\kappa^{n+2}c\La^{1-\theta}\iint_{G_{2\rho}^\la(z_0)\cap \Psi((4c)^{-1/\theta}\La)}H(z,|\na u|)^{\theta }\,dz\\
         &\qquad+2\kappa^{n+2}c\iint_{G_{2\rho}^\la(z_0)\cap \Theta((4c)^{-1}\La)}H(z,|F|) \,dz\,.
         \end{split}
     \end{align*}
     Replacing the constant $c$ above with $(4\kappa^{n+2}c)^\frac{1}{\theta_0}$, the proof is completed.
\end{proof}

\section{Proof of the main result}\label{sec:main-proof}
\subsection{Stopping time argument}\label{stopping time argument}
In this section, we prove that conditions \ref{p1}-\ref{p2} and \ref{q1}-\ref{q2} are satisfied under our regime. First of all, for any $\rho>0$ and $\hbar>1$, we denote the constant in Lemma~\ref{caccio_lem} for $r=\rho$, $R=2\rho$, $\tau^2=\hbar^{2-p}\rho^2$ and $\ell^2=\hbar^{2-p}(2\rho)^2$ by
\begin{align}\label{def_const_caccio}
    \ccc=\ccc(n,p,q,\nu,L)\,.
\end{align}
Let $r\in(0,1)$ and suppose $Q_{4r}(z_0)\subset\Om_T$. We define $\la_0$ and $\La_0$ as
\begin{align}\label{def_la}
		\begin{split}
		    \la^p_0
            &:=\frac{\|u\|_{L^\infty(\Om_T)}^p}{(4r)^p}+ \|a\|_{L^\infty(\Om_T)} \frac{\|u\|_{L^\infty(\Om_T)}^q}{(4r)^q}+ \left( \fiint_{Q_{4r}(z_0)} H(z,|F|)\,dz \right)^\frac{p}{2} +1
		\end{split}
\end{align}
and
\begin{align}\label{def_La}
    \La_0:=\la_0^p+\|a\|_{L^\infty(\Om_T)}\la_0^q\,.
\end{align}
Recall $K$ and $\cc>K^\frac{1}{p}$ from the previous section in \eqref{K_and_kappa},
\[
K=2+10\ca\|u\|_{L^\infty(\Om_T)}^{q-p}\quad\text{and}\quad \cv = 10(1+\cc^{q}\|u\|_{L^\infty(\Om_T)}^{q-p}+\ca +10\cc\ca )\,.
\]
Also, along with notation \eqref{upper_level_set} and \eqref{upper_F}, we write 
\begin{align*}
    \begin{split}
        &\Psi(\varkappa,\rho)=\Psi(\varkappa)\cap Q_{\rho}(z_0)=\{ z\in Q_{\rho}(z_0): H(z,|\na u(z)|)>\varkappa\}\,,\\
        &\Theta(\varkappa,\rho)={\Theta}(\varkappa)\cap Q_{\rho}(z_0)=\{ z\in Q_{\rho}(z_0): H(z,|F(z)|)>\varkappa\}\,.
    \end{split} 
\end{align*}

\begin{lemma}\label{La_breaking_lem}
Let $r\le r_1<r_2\le 2r$. Suppose
\[
\La>(4\ccc)^{ q }\left(\frac{8\cv r}{r_2-r_1}\right)^{\frac{q^2}{p}+\frac{q(n+2)}{2} }\La_0\,.
\]
For any $w\in Q_{4r}(z_0)$, let $\la_w$ be the unique positive number such that $H(w,\la_w)=\La$. Then we have
\[
\la_w> 4\ccc\left(\frac{8\cv r}{r_2-r_1}\right)^{ \frac{q}{p}+\frac{n+2}{2} }\la_0\,.
\]
\end{lemma}
\begin{proof}
    It can be proved by contradiction. Suppose the conclusion is false and
    \[
    \la_w\le 4\ccc\left(\frac{8\cv r}{r_2-r_1}\right)^{\frac{q}{p}+\frac{n+2}{2} }\la_0\,.
    \]
    Then, it is easy to see
    \[
    \La=H(w,\la_w) \le (4\ccc)^q\left(\frac{8\cv r}{r_2-r_1}\right)^{ \frac{q^2}{p}+\frac{q(n+2)}{2} }\La_0 <\La\,.
    \]
    Thus, the statement of this lemma is true.
\end{proof}

\begin{lemma}\label{p_stopping_lem}
Let $r\le r_1<r_2\le 2r$. Suppose $w\in \Psi(\La,r_1)$ for 
\[
\La>(4\ccc)^{ q }\left(\frac{8\cv r}{r_2-r_1}\right)^{\frac{q^2}{p}+\frac{q(n+2)}{2} }\La_0\,.
\]
Then for $\la_w$ defined as $H(w,\la_w)=\La$, there exists stopping time $\rho_{ w}\in(0,(r_2-r_1)/(2\cv))$ such that 
for all $s\in (\rho_w,r_2-r_1)$ we have
\[
	\fiint_{Q_{\rho_{ w}}^{\law}( w)} (H(z,|\na u|) + H(z,|F|) )\,dz= \law^p
\]
and 
\[
\fiint_{Q_s^{\law}( w)} (H(z,|\na u|) + H(z,|F|) ) \,dz<\law^p\,.
\]
\end{lemma}

\begin{proof}
    For any $s\in [(r_2-r_1)/(2\cv),r_2-r_1)$, we have by Lemma~\ref{caccio_lem} and \eqref{def_const_caccio} that
\begin{align} \label{lem_5_2_est}
	\begin{split}
		&\fiint_{Q_s^{\law}( w)} ( H(z,|\na u|)+ H(z,|F|) )\,dz\\
        &\le \ccc\fiint_{Q_{2s}^{\law}( w)} \left( \frac{|u-(u)_{Q_{2s}^{\law}(w)}|^p}{s^p} +a(z) \frac{|u-(u)_{Q_{2s}^{\law}(w)}|^q}{s^q} \right)\,dz
        \\
        &\qquad+\ccc \law^{p-2} \fiint_{Q_{2s}^{\law}(w)}\frac{|u-(u)_{Q_{2s}^{\law}(w)}|^2}{s^2}\,dz + \ccc \fiint_{Q_{2s}^{\law}(w)} H(z,|F|) \,dz\,.
	\end{split}
\end{align}
For the first term on the right hand side, we estimate as follows
\begin{align*}
    \begin{split}
        &\ccc\fiint_{Q_{2s}^{\law}( w)} \left( \frac{|u-(u)_{Q_{2s}^{\law}(w)}|^p}{s^p} +a(z) \frac{|u-(u)_{Q_{2s}^{\law}(w)}|^q}{s^q} \right)\,dz\\
        &\le 2^q\ccc \fiint_{Q_{2s}^{\law}( w)} \left( \frac{|u|^p}{s^p} +a(z) \frac{|u|^q}{s^q} \right)\,dz\\
        &\le 2^q \ccc\left( \frac{\|u\|_{L^\infty(\Om_T)}^p}{s^p}+\|a\|_{L^\infty(\Om_T)}\frac{\|u\|_{L^\infty(\Om_T)}^q}{s^q}  \right)\\
        &\le \ccc\left(  \frac{8\kappa r}{r_2-r_1} \right)^q\la_0^p\\
        &< \frac{1}{4}\law^p\,,
    \end{split}
\end{align*}
where to obtain the last inequality, we used Lemma~\ref{La_breaking_lem}. 
For the second integral on the right hand side of \eqref{lem_5_2_est}, we use the fact that $p\geq 2$ along with Lemma~\ref{La_breaking_lem} to have
\begin{align*}
    \begin{split}
        &\ccc \law^{p-2} \fiint_{Q_{2s}^{\law}(w)}\frac{|u-(u)_{Q_{2s}^{\law}(w)}|^2}{s^2}\,dz\\
        &\le \law^{p-2} 2^2\ccc  \fiint_{Q_{2s}^{\law}(w)}\frac{|u|^2}{s^2}\,dz\\
        &\le \law^{p-2} 2^2\ccc\frac{\|u\|_{L^\infty(\Om_T)}^2}{s^2}\\
        &\le \law^{p-2} \ccc\left(  \frac{ 8\kappa r }{r_2-r_1} \right)^2\la_0^2\\
        &< \frac{1}{4}\law^{p}\,.
    \end{split}
\end{align*}
And finally using $p\geq 2$ along with Lemma~\ref{La_breaking_lem}, we get
\begin{align*}
    \begin{split}
        \ccc\fiint_{Q_{2s}^{\law}(w)} H(z,|F|) \,dz
        &\le \law^{p-2} \ccc\left( \frac{8\kappa r}{r_2-r_1} \right)^{n+2}\fiint_{Q_{4r}(z_0)} H(z,|F|) \,dz\\
        &\le \law^{p-2} \ccc\left( \frac{8\kappa r}{r_2-r_1} \right)^{n+2}\la_0^2\\
        &< \frac{1}{4}\law^p\,.
    \end{split}
\end{align*}
Combining these estimates, we have
\[
\fiint_{Q_s^{\law}( w)} ( H(z,|\na u|)+ H(z,|F|) )\,dz<\law^p
\]
for all $s\in [ (r_2-r_1)/(2\kappa), r_2-r_1 )$.
On the other hand, since $w\in \Psi(\La,r_1)$ holds, we deduce from the Lebesgue point theorem and continuity of integral with respect to the radius that there exists $\rho_w$ satisfying the statement of this lemma.
\end{proof}
Note that if $\rho_w$ further satisfies
\begin{align}\label{p_intr_ineq}
    K\ge \left( \frac{\|u\|_{L^\infty(\Om_T)}}{\rho_w} \right)^{q-p}\sup_{z\in Q_{4\rho_w}^{\law}(w)}a(z)\,,
\end{align}
then it is $p$-intrinsic at $w$. Along with stopping time argument in the previous lemma, \ref{p1}-\ref{p2} are verified. 
Before we consider the other case, we rewrite Lemma~\ref{lemma_decay} and Lemma~\ref{p_intr} as follows.
\begin{lemma}\label{p_decay_lem}
Let $r\le r_1<r_2\le 2r$.
    Suppose $w\in \Psi(\La,r_1)$.
    Under the same assumptions in Lemma~\ref{p_stopping_lem}, if \eqref{p_intr_ineq} holds, then there exists $\cc=\cc(\data)\ge K^\frac{1}{p}$ such that
    \[
    \rho_w\le \cc\law^{-1}\quad\text{and}\quad \sup_{z\in Q_{4\rho_w}^{\law}(w)}a(z)\le \cc^q\|u\|_{L^\infty(\Om_T)}^{q-p}\law ^{p-q}\,.
    \]
\end{lemma}

On the other hand, we suppose \eqref{p_intr_ineq} fails.
\begin{lemma}\label{comp_a_lem}
Let $r\le r_1<r_2\le 2r$.
    Suppose $w\in \Psi(\La,r_1)$.
    Under the same assumptions in Lemma~\ref{p_stopping_lem}, if there holds
    \[
K< \left( \frac{\|u\|_{L^\infty(\Om_T)}}{\rho_w} \right)^{q-p}\sup_{z\in Q_{4\rho_w}^{\law}(w)}a(z)\,,
\]
then, we have
\[
\frac{a(w)}{2\ca}< a(z)< 2\ca a(w)\quad\text{
for all $z\in Q_{5\rho_w}(w)$}\,.
\]
\end{lemma}

\begin{proof}
    Note that it is enough to show
    \begin{align}\label{q_intr_claim}
        2\ca(5\rho_w)^\alpha<\sup_{z\in Q_{5\rho_w}(w)}a(z)\,.
    \end{align}
    Indeed, it follows from \eqref{def_holder} that
    \[
    2\ca(5\rho_w)^\alpha< \sup_{z\in Q_{5\rho_w}(w)}a(z)< \ca \left( \inf_{z\in Q_{5\rho_w}(w)}a(z) +(5\rho)^\alpha \right)
    \]
    and therefore, we have
    \[
    (5\rho_w)^\alpha<\inf_{z\in Q_{5\rho_w}(w)}a(z)\,.
    \]
    Again using the hypothesis on $a$, we have
    \[
    \sup_{z\in Q_{5\rho_w}(w)}a(z)\le \ca \left( \inf_{z\in Q_{5\rho_w}(w)}a(z)+(5\rho_w)^\alpha \right) < 2\ca\inf_{z\in Q_{5\rho_w}(w)}a(z)\,.
    \]
    The conclusion follows from the above inequality. We prove \eqref{q_intr_claim} by making contradiction. Suppose contrary
    \[
    2\ca(5\rho_w)^\alpha\ge\sup_{z\in Q_{5\rho_w}(w)}a(z)\,.
    \]
    Then since $Q_{4\rho_w}^{\law}(w)\subset Q_{5\rho_w}(w)$ holds, combining it with the assumption in this lemma, we obtain
    \[
    K<\left( \frac{\|u\|_{L^\infty(\Om_T)}}{\rho_w} \right)^{q-p}2\ca (5\rho_w)^\alpha\le 10\ca\|u\|_{L^\infty(\Om_T)}^{q-p}\rho_w^{p+\alpha-q}\,.
    \]
    Recalling $\rho_w\le r_2-r_1\le r\le 1$ and $q\leq p+\alpha$, the above inequality is a contradiction by the construction of $K$ in \eqref{K_and_kappa},
\end{proof}

We now prove the stopping time argument with the $(p,q)$-intrinsic cylinder.
\begin{lemma}%\label{q_stopping_lem}
Let $r\le r_1<r_2\le 2r$.
    Suppose $w\in \Psi(\La,r_1)$.
    Under the same assumptions in Lemma~\ref{comp_a_lem}, there exists stopping time $\varsigma_{ w}\in(0,\rho_w)$ such that
\[
	\fiint_{G_{\varsigma_{ w}}^{\law}( w)}  ( H(z,|\na u|) +H(z,|F|)   )  \,dz= \La
\]
and 
for all $s\in (\varsigma_w,r_2-r_1)$ it holds
\[
\fiint_{G_s^{\law}( w)}  ( H(z,|\na u|) +H(z,|F|)   )  \,dz<\La\,.
\]
\end{lemma}
\begin{proof}
    By the previous lemma, we have $a(w)>0$ and thus $G_s^{\law}(w)\Subset Q_s^{\law}(w)$ for any $s>0$.
    Now for $s\in [\rho_w,r_2-r_1)$, we have from Lemma~\ref{p_stopping_lem} that
    \begin{align*}
        \begin{split}
            &\fiint_{G_s^{\law}(w)}  ( H(z,|\na u|) +H(z,|F|)   ) \,dz\\
            &< \frac{|Q_{s}^{\law}(w)|}{|G_{s}^{\law}(w)|}\fiint_{Q_{s}^{\law}(w)} ( H(z,|\na u|) +H(z,|F|)   ) \,dz\\
            &\le \frac{\La}{\law^p}\law^p=\La\,.
        \end{split}
    \end{align*}
    Again by the Lebesgue point theorem and continuity of integral with respect to the radius, we find the stopping time $\varsigma_w\in (0,\rho_w)$ satisfying the claim of this lemma.
\end{proof}

We have proved that if \eqref{p_intr_ineq} fails, then \ref{q1}-\ref{q2} hold.

\subsection{Vitali type covering argument} Recall the stopping time parameter $\La$ for $(p,q)$-scenario, cf.~\ref{q2}.
Since intrinsic geometries may vary from point to point in $\Psi(\La,r_1)$, standard Vitali covering lemma cannot be directly applied. In this subsection, we modify the original proof to show covering lemma in our setting.

For each $ w\in \Psi(\La,r_1)$, where $r\le r_1<r_2\le 2r$ and 
\[
\La>(4\ccc)^{q}\left(\frac{8\cv r}{r_2-r_1}\right)^{ \frac{q^2}{p}+\frac{q(n+2)}{2}}\La_0\,.
\] we denote
\[
\mQ ( w)= 
\begin{cases}
Q_{2\rho_{ w}}^{\law}( w)&\text{if $p$-intrinsic case\,,}\\
G_{2\varsigma_{ w}}^{\law}( w)&\text{if $(p,q)$-intrinsic case\,.}
\end{cases}
\]
For the simplicity, we also write
\[
l_{ w}=\begin{cases}
		2\rho_{w}&\text{if $p$-intrinsic case\,,}\\
		2\varsigma_{w}&\text{if $(p,q)$-intrinsic case\,.}
	\end{cases}
\]
We will prove a Vitali type covering lemma from the following set of cylinders
\[
	\mathcal{F}=\left\{\mQ ( w):  w\in \Psi(\La,r_1)\right\}\,.
\]
In order to do this, we need the comparability of the scaling factors $\la_{(\cdot)}$ in the neighborhood.
\begin{lemma}\label{la_comp_lem}
    Suppose $U(w), U(v)\in \mathcal{F}$ with $U(w)\cap U(v)\ne \emptyset$ and $l_w\le 2l_v$. Then for $\cc\ge 2$ defined in Lemma~\ref{p_decay_lem}, we have
    \[
    \lav\le (2(1+\ca+10\cc\ca))^\frac{1}{p}\law\,.
    \]
    Moreover, if $U(v)$ is $(p,q)$-intrinsic, then we also have
    \[
    \law\le (2(1+\ca+10\cc\ca))^\frac{1}{p}\lav\,.
    \]
\end{lemma}
\begin{proof}
By the assumption, we observe that $Q_{l_w}(w)\cap Q_{l_v}(v)\ne \emptyset$. By the standard Vitali covering lemma, there holds
\begin{align}\label{stndr_vitali}
    Q_{l_{w}}(w) \subset Q_{5l_{v}}(v)\,.
\end{align}
We will prove the first statement of this lemma by contradiction. Suppose we have
\begin{align}\label{la_comp_contr}
     \lav > (2(1+\ca+10\cc\ca))^\frac{1}{p}\law\,.
\end{align}
 Now, we divide cases:
 \begin{enumerate}[label=(\roman*)]
     \item\label{i} $\mQ (v)=Q_{l_v}^{\la_v}(v)$,
     \item\label{ii} $\mQ (v)=G_{l_v}^{\la_v}(v)$.
 \end{enumerate}

 \textit{Case} \ref{i}: Since we have $\La=H(w,\law)$ and \eqref{stndr_vitali} holds, we apply \eqref{def_holder} to have
 \[
0< \La\le\law^p+\ca( a(v)+(5l_v)^\alpha)\law^q\,.
 \]
 By using \eqref{la_comp_contr}, we have
 \[
 \La \le\frac{1}{2(1+\ca+10\cc\ca)}(\lav^p+\ca( a(v)+(5l_v)^\alpha)\lav^q )\,.
 \]
  As $U(v)$ is a $p$-intrinsic cylinder with $l_v=2\rho_v$, it follows from Lemma~\ref{p_decay_lem} that
 \[
 \La \le\frac{1}{2(1+\ca+10\cc\ca)}((1+10\ca\cc)\lav^p+\ca a(v)\lav^q )\le \frac{1}{2}\La\,,
 \]
where we used $\lav^{q-\alpha}\le \lav^p$, since $\lav>1$ and $q\leq p+\alpha$. Hence, it is a contradiction.

 \textit{Case} \ref{ii}: In this case, we have $l_v=\varsigma_v<\rho_v$ and thus by Lemma~\ref{comp_a_lem} and \eqref{stndr_vitali}, there holds
 \begin{align}\label{comp_a_1}
     \frac{a(v)}{2\ca}<a(w)<2\ca a(v)\,.
 \end{align}
Then from \eqref{la_comp_contr} and the above display, we get
\[
\La=\law^p+a(w)\law^q\le (1+2\ca) (\law^p+a(v)\law^q)\le \frac{1+2\ca}{2(1+\ca+10\cc\ca)}(\lav^p+a(v)\lav^q)<\La\,.
\]
Again, we arrived at the contradiction and the conclusion of the first claim is true. 

To prove the second statement, we assume 
\[
\law > (2(1+\ca+10\cc\ca))^\frac{1}{p}\lav.
\]
Then along with \eqref{comp_a_1}, we have
\[
\La=\lav^p+a(v)\lav^q\le (1+2\ca)(\lav^p+a(w)\lav^q)\le \frac{1+2\ca}{2(1+\ca+10\cc\ca)}(\law^p+a(w)\law^q)<\La\,.
\]
Therefore, the second statement also holds.
\end{proof}

\begin{lemma}\label{vitali_lem}
    Let $\mathcal{F}$ be defined as above and $\kappa$ is in \eqref{K_and_kappa}, Then there exists a pairwise disjoint countable subcollection $\mathcal{G}$ of $\mathcal{F}$ such that for any $U(w)\in \mathcal{F}$, there exists $U(v)\in \mathcal{G}$ with
    \[
    U(w)\subset \kappa U(v)\,.
    \]
\end{lemma}

\begin{proof}
    Since $l_w\le 2\rho_w\le 2r$, for each $j\in\mathbb N$, we define
    \[
	\mathcal{F}_j=\left\{\mQ ( w)\in \mathcal{F}: \frac{2r}{2^j}<l_{ w}\le\frac{2r}{2^{j-1}} \right\}\,.
    \]
  Next, we consider subcollections $\mathcal{G}_j\subset \mathcal{F}_j$ for each $j\in\mathbb N$ recursively. Let $\mathcal{G}_1$ be a maximal disjoint collection of cylinders in $\mathcal{F}_1$. With chosen $\mathcal{G}_1,...,\mathcal{G}_{k}$, we take
\[
	\mathcal{G}_{k+1}=\Bigl\{ \mQ ( w)\in \mathcal{F}_k: \mQ ( w)\cap \mQ ( v)=\emptyset\quad\text{for all}\quad\mQ ( v)\in \bigcup_{j=1}^{k}\mathcal{G}_j\Bigr\}\,.    
\]
Note that if $U(w)\in \mathcal{F}_j$ for some $j$, then 
\[
|U(w)|\ge|G_{l_w}^{\law}(w)|=2|B_1|\frac{\law^2}{\La}l_w^{n+2}> 2^{-(j-1)(n+2)}\La^{-1}|B_1|r^{n+2}\,.
\]
Therefore, since $U(w)\subset Q_{2r}(z_0)$ holds for all $w$, the cardinality of $\mathcal{G}_j$ are finite. We now claim that
\[
\mathcal{G}=\bigcup_{j=1}^\infty\mathcal{G}_j
\]
is the desired subcollection. As it is constructed to be pairwise disjoint, it remains to show the covering property. 

For any $\mQ ( w)\in \mathcal{F}$, there exists $j \in \mathbb{N}$ such that $U ( w)\in \mathcal{F}_j$. Moreover, by the construction of $\mathcal{G}_j$, there exists a cylinder 
$U ( v)\in\cup_{i=1}^j \mathcal{G}_i$ such that $U(w)\cap U(v)\ne \emptyset$. Besides, we have 
\begin{align}\label{radi_comp}
    l_w\le 2l_v\,.
\end{align}
As in the proof in the previous lemma, we again have
\[
Q_{l_w}(w)\subset Q_{5l_v}(v)\,.
\]
Thus, denoting $w=(x,t)$ and $v=(y,s)$ for $x,y\in \mathbb{R}^n$ and $t,s\in \mathbb{R}$, we have
\[
B_{l_w}(x)\subset B_{5l_v}(y)\subset B_{\kappa l_v}(y)\,.
\]
Therefore, it is enough to prove the inclusion of the time interval part. Since the time interval of $U(w)$ and the time interval of $U(v)$ intersect, it is enough to show that 
\[
|\text{time interval of $U(w)$}|\le \frac{\kappa^2}{2}|\text{time interval of $U(v)$}|\,.
\]
We prove it by dividing cases:
\begin{enumerate}[label=(\alph*)]
\item\label{a} $\mQ (v)=Q_{l_v}^{\la_v}(v)$ and $\mQ (w)=Q_{l_w}^{\la_w}(w)$,
\item\label{b} $\mQ (v)=Q_{l_v}^{\la_v}(v)$ and $\mQ (w)=G_{l_w}^{\la_w}(w)$,
\item\label{c} $\mQ (v)=G_{l_v}^{\la_v}(v)$ and $ \mQ (w)=G_{l_w}^{\la_w}(w)$,
\item\label{d} $\mQ (v)=G_{l_v}^{\la_v}(v)$ and $ \mQ (w)=Q^{\la_w}_{l_w}(w)$.
\end{enumerate} 
\textit{Case} \ref{a}: Using Lemma~\ref{la_comp_lem} and \eqref{radi_comp}, we have
\begin{align*}
    \begin{split}
        2|I_{l_w}^{\law}(t)|
    &=2^2\law^{2-p}l_w^2\\
    &\le 2^2(2(1+\ca+10\cc\ca))^\frac{p-2}{p}\lav^{2-p}(2l_v)^2\\
    &= 2^4(2(1+\ca+10\cc\ca))^\frac{p-2}{p}|I_{l_v}^{\lav}(s)|\,.
    \end{split}
\end{align*}
Since we have set $2^4(2(1+\ca+10\ca\cc))^\frac{p-2}{p}\le \kappa^2$, we get
\[
2|I_{l_w}^{\law}(t)|\le \kappa^2|I_{l_v}^{\lav}(s)|=|I_{\kappa l_v}^{\lav}(s)|\,.
\]

\textit{Case} \ref{b}: As in the previous case, we again have
\[
2|J_{l_w}^{\law}(t)|\le 2|I_{l_w}^{\law}(t)|\le |I_{\kappa l_v}^{\lav}(s)|\,.
\]

\textit{Case} \ref{c}: We apply Lemma~\ref{la_comp_lem} and \eqref{radi_comp} to get
\begin{align*}
    \begin{split}
        2|J_{l_w}^{\law}(t)|
    &=2^2\frac{\la_w^2}{\La}l_w^2\\
    &\le 2^2(2(1+\ca+10\cc\ca))^\frac{2}{p}\frac{\la_v^2}{\La}(2l_v)^2\\
    &=2^4(2(1+\ca+10\cc\ca))^\frac{2}{p}|J_{l_v}^{\lav}(s)|\\
    &\le |J_{\kappa l_v}^{\lav}(s)|\,.
    \end{split}
\end{align*}

\textit{Case} \ref{d}: To begin with, note that Lemma~\ref{p_decay_lem} gives
\begin{align*}
    \begin{split}
        \law^{2-p}
        &=\frac{(1+\cc^q\|u\|_{L^\infty(\Om_T)}^{q-p}\law^2}{(1+\cc^q\|u\|_{L^\infty(\Om_T)}^{q-p})\law^p}\\
        &\le (1+\cc^q\|u\|_{L^\infty(\Om_T)}^{q-p})\frac{\la_w^2}{\law^p+a(w)\law^q}\\
        &\le (1+\cc^q\|u\|_{L^\infty(\Om_T)}^{q-p})\frac{\la_w^2}{\La}\,.
    \end{split}
\end{align*}
Therefore, applying Lemma~\ref{la_comp_lem} and \eqref{radi_comp}, we obtain
\begin{align*}
    \begin{split}
        2|I_{l_w}^{\law}(t)|
        &=2^2\law^{2-p}l_w^2\\
        &\le 2^2 (2(1+\cc^q\|u\|_{L^\infty(\Om_T)}^{q-p}+\ca +10\cc\ca ))^{1+\frac{2}{p}}\frac{\lav^2}{\La}(2l_v)^2\\
        &= 2^4 (2(1+\cc^q\|u\|_{L^\infty(\Om_T)}^{q-p}+\ca +10\cc\ca ))^{1+\frac{2}{p}}|J_{ l_v}^{\lav}(s)|\\
        &\le|J_{\kappa l_v}^{\lav}(s)|\,.
    \end{split}
\end{align*}
This completes the proof.
\end{proof}

\subsection{Proof of Theorem~\ref{main_theorem}}
Recall we have set $r\le r_1<r_2\le 2r$ with $r\in(0,1)$.
For selected subcollection $\mathcal{G}$ in Lemma~\ref{vitali_lem}, we simply write
\[
 \mathcal{G}=\{ \mQ_i \}_{1\le i \le \infty}\,,
\]
where $\mQ_i=\mQ(w_i)$ for some $w_i\in \Psi(\La,r_1)$. Depending on the intrinsic geometry of $\mQ_i\in \mathcal{G}$, it follows by Proposition~\ref{p_est_vitali} and Proposition~\ref{q_est_vitali} that for any $i\in \mathbb{N}$
\begin{align*}
    \begin{split}
        \iint_{\cv\mQ _{i}}H(z,|\na u|)\,dz
		&\le c\La^{1-\theta}\iint_{\mQ _i\cap \Psi(c^{-1}\La)}H(z,|\na u|)^\theta\,dz\\
        &\qquad + c\iint_{U_i\cap \Theta(c^{-1} \la) } H(z,|F|)\,dz\,,
    \end{split}
\end{align*}
where $c=c(\data)>1$ is a constant and fixed $\theta\in (0,1)$.
Employing disjointedness and covering property in Lemma~\ref{vitali_lem}, we get
\begin{align*}
    \begin{split}
        \iint_{\Psi(\La,r_1)}H(z,|\na u|)\,dz
&\le\sum_{i=1}^\infty\iint_{\cv\mQ _{i}}H(z,|\na u|)\,dz\\
&\le c\La^{1-\theta}\sum_{i=1}^\infty\iint_{\mQ _i\cap \Psi(c^{-1}\La)}H(z,|\na u|)^\theta\,dz\\
&\qquad +c \sum_{i=1}^\infty\iint_{\mQ _i\cap \Theta(c^{-1}\La)}H(z,|F|)\,dz\\
&\le c\La^{1-\theta}\iint_{\Psi(c^{-1}\La,r_2)}H(z,|\na u|)^\theta\,dz\\
&\qquad+ c\iint_{\Theta(c^{-1}\La,r_2)}H(z,|F|)\,dz\,.
    \end{split}
\end{align*}
On the other hand, since we have
\[
\iint_{\Psi(c^{-1}\La,r_1)\setminus \Psi(\La,r_1)}H(z,|\na u|)\,dz
\le\La^{1-\theta}\iint_{\Psi(c^{-1}\La,r_2)}H(z,|\na u|)^{\theta}\,dz\,,
\]
we deduce that
\begin{align}\label{covering_est_1}
\begin{split}
    \iint_{\Psi(c^{-1}\La,r_1)}H(z,|\na u|)\,dz
    &\le c\La^{1-\theta}\iint_{\Psi(c^{-1}\La,r_2)}H(z,|\na u|)^\theta\,dz\\
    &\qquad +c \iint_{\Theta(c^{-1}\La,r_2)}H(z,|F|)\,dz\,,
\end{split}
\end{align}
where $c=c(data)>1$ is a constant.
We now take $k\in\mathbb N$ and consider
\[		
H(z,|\na u|)_k=\min\{H(z,|\na u|),k\}
\]
and
\[
\Psi_k(\La,\rho)=\{z\in Q_{\rho}(z_0):H(z,|\na u(z)|)_k>\La\}\,.
\]
It is easy to see that if $\La>k$, then $\Psi_k(\La,\rho)=\emptyset$, and if $\La\le k$, then $\Psi_k(\La,\rho)=\Psi(\La,\rho)$. Therefore, along with these notations, \eqref{covering_est_1} becomes
\begin{align*}
	\begin{split}
		\iint_{\Psi_k(c^{-1}\La,r_1)}\left(H(z,|\na u|)_k\right)^{1-\theta}H(z,|\na u|)^\theta\,dz
		&\le c\La^{1-\theta}\iint_{\Psi_k(c^{-1}\La,r_2)}H(z,|\na u|)^\theta\,dz\\
        &\qquad +c \iint_{\Theta(c^{-1}\La,r_2)}H(z,|F|)\,dz
	\end{split}
\end{align*}
for any 
\[
\La>(4\ccc)^{q}\left(\frac{8\cv r}{r_2-r_1}\right)^{\frac{q^2}{p}+\frac{q(n+2)}{2} }\La_0\,.
\]
Denoting
\[
	\La_1= c^{-1}(4\ccc)^q\left(\frac{8\cv r}{r_2-r_1}\right)^{\frac{q^2}{p}+\frac{q(n+2)}{2}}\La_0\,,
\]
for any $\La>\La_1$ we obtain
\begin{align*}
	\begin{split}
		&\iint_{\Psi_k(\La,r_1)}\left(H(z,|\na u|)_k\right)^{1-\theta}H(z,|\na u|)^\theta\,dz\\
		&\le c\La^{1-\theta}\iint_{\Psi_k(\La,r_2)}H(z,|\na u|)^\theta\,dz+c\iint_{\Theta(\La,2r)}H(z,|F|)\,dz\,.
	\end{split}
\end{align*}

Let $\ep\in(0,1)$ will be determined later. We multiply the inequality above by $\La^{\ep-1}$ and integrate over $(\La_1,\infty)$. Then, we get
\begin{align*}
	\begin{split}
		\mathrm{I}&=\int_{\La_1}^{\infty}\La^{\ep-1}\iint_{\Psi_k(\La,r_1)}\left(H(z,|\na u|)_k\right)^{1-\theta}H(z,|\na u|)^\theta\,dz\,d\La\\
		&\le c\int_{\La_1}^{\infty}\La^{\ep-\theta}\iint_{\Psi_k(\La,r_2)}H(z,|\na u|)^\theta\,dz\,d\La  +  c \int_{\La_1}^{\infty}\La^{\ep-1} \iint_{\Theta(\La,2r)}H(z,|F|)\,dz \, d\La \\
		&= \mathrm{II} + \mathrm{III}\,.
	\end{split}
\end{align*}
Applying Fubini's theorem to $\mathrm{I}$, it follows that
\begin{align*}
	\begin{split}
		\mathrm{I}
		&=\frac{1}{\ep}\iint_{\Psi_k(\La_1,r_1)}\left(H(z,|\na u|)_k\right)^{1-\theta+\ep}H(z,|\na u|)^\theta\,dz\\
		&\qquad-\frac{1}{\ep}\La_1^\ep\iint_{\Psi_k(\La_1,r_1)}\left(H(z,|\na u|)_k\right)^{1-\theta}H(z,|\na u|)^\theta\,dz\,.
	\end{split}
\end{align*}
Meanwhile, since we have
\begin{align*}
	\begin{split}
		&\iint_{Q_{r_1}(z_0)\setminus \Psi_k(\La_1,r_1)}\left(H(z,|\na u|)_k\right)^{1-\theta+\ep}H(z,|\na u|)^\theta\,dz\\
		&\le \La_1^{\ep}\iint_{Q_{2r}(z_0)}\left(H(z,|\na u|)_k\right)^{1-\theta}H(z,|\na u|)^\theta\,dz\,,
	\end{split}
\end{align*}
we obtain
\begin{align*}
	\begin{split}
		\mathrm{I}\ge& \frac{1}{\ep}\iint_{Q_{r_1}(z_0)}\left(H(z,|\na u|)_k\right)^{1-\theta+\ep}H(z,|\na u|)^\theta\,dz\\
		&\qquad-\frac{2}{\ep}\La_1^\ep\iint_{Q_{2r}(z_0)}\left(H(z,|\na u|)_k\right)^{1-\theta}H(z,|\na u|)^\theta\,dz\,.
	\end{split}
\end{align*}
Similarly, by Fubini's theorem, we have
\[
		\mathrm{II}
		\le\frac{1}{1-\theta+\ep}\iint_{Q_{r_2}(z_0)}\left(H(z,|\na u|)_k\right)^{1-\theta+\ep}H(z,|\na u|)^\theta \,dz
\]
and
\[
\mathrm{III}\le \frac{1}{\ep}\iint_{Q_{2r}(z_0)}H(z,|F|)^{1+\ep} \,dz\,.
\]
Combining the estimates above, we obtain
\begin{align*}
	\begin{split}
		&\iint_{Q_{r_1}(z_0)}\left(H(z,|\na u|)_k\right)^{1-\theta+\ep}H(z,|\na u|)^\theta\,dz\\
		&\le \frac{c\ep}{1-\theta+\ep}\iint_{Q_{r_2}(z_0)}\left(H(z,|\na u|)_k\right)^{1-\theta+\ep}H(z,|\na u|)^\theta \,dz\\
		&\qquad+c\La_1^\ep\iint_{Q_{2r}(z_0)}\left(H(z,|\na u|)_k\right)^{1-\theta}H(z,|\na u|)^\theta\,dz \\
        &\qquad + c\iint_{Q_{2r}(z_0)}H(z,|F|)^{1+\ep} \,dz
	\end{split}
\end{align*}
for $c=c(\data)$ and $\theta=\theta(\data)\in(0,1)$. We choose $\ep_0=\ep_0(\data)\in(0,1)$ so that for any $\ep\in(0,\ep_0)$,
\[
	\frac{c\ep}{1-\theta+\ep}\le\frac{1}{2}\,.
\]
Then, there holds
\begin{align*}
	\begin{split}
		&\iint_{Q_{r_1}(z_0)}\left(H(z,|\na u|)_k\right)^{1-\theta+\ep}H(z,|\na u|)^\theta\,dz\\
		&\le \frac{1}{2}\iint_{Q_{r_2}(z_0)}\left(H(z,|\na u|)_k\right)^{1-\theta+\ep}H(z,|\na u|)^\theta \,dz\\
		&\qquad+c\La_1^\ep\iint_{Q_{2r}(z_0)}\left(H(z,|\na u|)_k\right)^{1-\theta}H(z,|\na u|)^\theta\,dz \\
        &\qquad+ c\iint_{Q_{2r}(z_0)}H(z,|F|)^{1+\ep} \,dz\,.
	\end{split}
\end{align*}
Recalling $r\le r_1<r_2\le 2r$, we conclude from Lemma~\ref{tech_lem} that
\begin{align*}
	\begin{split}
		&\iint_{Q_{r}(z_0)}\left(H(z,|\na u|)_k\right)^{1-\theta+\ep}H(z,|\na u|)^\theta\,dz\\
		&\le c\La_0^\ep\iint_{Q_{2r}(z_0)}\left(H(z,|\na u|)_k\right)^{1-\theta}H(z,|\na u|)^\theta\,dz\\
        &\qquad+ c\iint_{Q_{2r}(z_0)}H(z,|F|)^{1+\ep} \,dz\,.
	\end{split}
\end{align*}
Letting $k$ go to $\infty$ and recalling \eqref{def_la} and $\eqref{def_La}$, we have 
\begin{align*}
    \begin{split}
        &\fiint_{Q_{r}(z_0)}H(z,|\na u|)^{1+\epsilon}\,dz\\
		&\le c\left(\frac{\|u\|^p_{\infty}}{r^p} +\|a\|_{\infty} \frac{\|u\|^q_{\infty}}{r^q}  + 
 \left( \fiint_{Q_{4r}(z_0)}H(z,|F|)\,dz \right)^\frac{p}{2} +1\right)^{\frac{q\epsilon}{p}}\\
        &\qquad\times \iint_{Q_{2r}(z_0)} H(z,|\na u|) \,dz+ c\fiint_{Q_{2r}(z_0)}H(z,|F|)^{1+\ep} \,dz\,,
    \end{split}
\end{align*}
where we abbreviated $\| \cdot \|_{\infty}= \| \cdot \|_{L^\infty(\Om_T)}$.
Finally, applying Lemma~\ref{caccio_lem} with $Q_{2r}(z_0)$ and $Q_{4r}(z_0)$, we obtain
\begin{align*}
    \begin{split}
        &\fiint_{Q_{r}(z_0)}H(z,|\na u|)^{1+\epsilon}\,dz\\
		&\le c\left(\frac{\|u\|^p_{\infty}}{r^p} +\|a\|_{\infty} \frac{\|u\|^q_{\infty}}{r^q}  + 
 \left( \fiint_{Q_{4r}(z_0)}H(z,|F|)\,dz \right)^\frac{p}{2} +1\right)^{\frac{q\epsilon}{p}}\\
        &\qquad\times \left(\frac{\|u\|^p_{\infty}}{r^p} +\|a\|_{\infty} \frac{\|u\|^q_{\infty}}{r^q} + \frac{\|u\|^2_{\infty}}{r^2}  +\fiint_{Q_{4r}(z_0)}H(z,|F|)\,dz \right)\\
        &\qquad+ c\fiint_{Q_{2r}(z_0)}H(z,|F|)^{1+\ep} \,dz\,,
    \end{split}
\end{align*}
where $c=c(\data,\|a\|_{L^\infty(\Om_T)})$.
Since $p\ge2$, we apply Young's inequality to the term with exponent $2$ on the right hand side in order to have
\begin{align*}
    \begin{split}
        \fiint_{Q_{r}(z_0)}H(z,|\na u|)^{1+\epsilon}\,dz
	&\le c\left(\frac{\|u\|^p_{\infty}}{r^p} +\|a\|_{\infty}\frac{\|u\|^q_{\infty}}{r^q}  +1\right)^{1+\frac{q\epsilon}{p}} \\
    &\qquad+ c\left( \fiint_{Q_{4r}(z_0)}H(z,|F|)^{1+\ep} \,dz\right)^{1+\frac{q}{2}}\,.
    \end{split}
\end{align*}
The proof is completed. \qed

\bibliography{parabolic_double_phase_bib}
\bibliographystyle{abbrv}

\end{document}